\documentclass[11pt,twoside, a4paper, english, reqno]{amsart}
\usepackage[dvips]{epsfig}
\usepackage{amscd}
\usepackage{amssymb}
\usepackage{amsthm}
\usepackage{amsmath}
\usepackage{latexsym}

\usepackage{upref}
\usepackage{bm}
\usepackage{color}
\usepackage{hyperref}
\usepackage{graphicx}
\usepackage{lipsum}

\hypersetup{linkcolor=blue, colorlinks=true
,citecolor = red}
\theoremstyle{plain}
\newtheorem{Th}{Theorem}[section]
\newtheorem{Lem}[Th]{Lemma}

\newtheorem{Prop}[Th]{Proposition}

 \theoremstyle{definition}
\newtheorem{Def}[Th]{Definition}

\newtheorem{Rem}[Th]{Remark}

\newtheorem{lema}{Lemma}

\newtheorem{Ass}{Assumption}

\newcommand{\be} {\begin{equation}}
\newcommand{\ee} {\end{equation}}

\newcommand{\R} {\mathbb{R}}

\newcommand{\rt}{\mathbb{R}^{2}}

\newcommand{\f}{\mathcal{F}}

\newcommand{\et}{\bm{\eta}}
\newcommand{\si}{\sum_{i=1}^n}
\newcommand{\sj}{\sum_{j=1}^n}
\newcommand{\inte}{\int_{\rt}}
\newcommand{\rk}{\bm{\rho}_{\tau}^k}
\newcommand{\rka}{\bm{\rho}_{\tau}^{k-1}}
\newcommand{\rki}{\rho_{\tau,i}^k}

\newcommand{\rkm}{\bm{\rho}^{k_m}_{\tau_m}}
\newcommand{\rkam}{\bm{\rho}^{k_m-1}_{\tau_m}}

\newcommand{\vr}{\bm{\varrho}}
\newcommand{\vri}{\varrho_i}
\newcommand{\vrj}{\varrho_j}

\newcommand{\rl}{\bm{\rho}_{\tau}^l}

\newcommand{\rti}{\rho_{\tau,i}}

\newcommand{\vu}{\Upsilon}
\newcommand{\fl}{\si\inte \rho_i(x) \ln \rho_i(x) \ dx + \si \sj \frac{a_{ij}}{4\pi}\inte \inte \rho_i(x) \ln |x-y| \rho_j(y) \ dx dy}

\newcommand{\G}{ \mathcal{G}_{\bm{\eta}}(\bm{\rho})}
\newcommand{\F}{\mathcal{F}(\bm{\rho})}
\newcommand{\ga}{\Gamma^{\bm{\beta}}_2}
\newcommand{\dw}{\bm{\mbox{d}}_{\bm{w}}}
\newcommand{\dws}{\mbox{d}_{w}}

\newcommand{\dwo}{\bm{\mbox{d}}_{\bm{w1}}}
\newcommand{\dwso}{\mbox{d}_{w1}}

\newcommand{\authorfootnotes}{\renewcommand\thefootnote{\@fnsymbol\c@footnote}}%

\numberwithin{equation}{section} \allowdisplaybreaks

\title[Patlak-Keller-Segel system]{On the critical mass Patlak-Keller-Segel system for multi-species populations: global existence and infinite time aggregation}

\author[D. Karmakar]{Debabrata Karmakar}
\address[D. Karmakar]{Tata Institute of Fundamental Research,
Centre For Applicable Mathematics,
Post Bag No 6503, GKVK Post Office,
Sharada Nagar, Chikkabommsandra,
Bangalore 560065,
Karnataka, India}
\email{\tt debabrata@tifrbng.res.in}

\author[G. Wolansky]{Gershon Wolansky}
\address[G. Wolansky]{Technion, Israel Institute of Technology, 32000 Haifa, Israel}
\email{\tt gershonw@math.technion.ac.il}


\begin{document}

\begin{abstract}
We study the global in time existence and long time asymptotic of solutions to the parabolic-elliptic Patlak-Keller-Segel system for the multi-species populations in the whole Euclidean space $\rt.$ 
We prove that at the borderline case of {\it critical mass} there exists a global {\it free energy solution} subject to initial data with finite entropy and second moment. Moreover, we show that as time $t$ approaches to infinity, all the components of the solutions concentrate in the form of a Dirac measure at a single point.
Our approach utilizes the gradient flow structure in Wasserstein space in the spirit of De Giorgi's minimizing movement or the JKO-schemes. Due to the critical mass, the minimization problem in JKO-schemes may not admit a solution in general.  We find a necessary and sufficient criterion for which any minimizing sequence 
remains uniformly bounded in an appropriate topology to ensure the existence of a minimizer.
\end{abstract}

\maketitle
{\small\tableofcontents}

\section{Introduction}  

In this article, we study the global in time existence of solutions to
the parabolic-elliptic Patlak-Keller-Segel system (henceforth abbreviated PKS-system)  for the multi-species populations at the critical mass regime 
(to be defined in a moment). Multi-species PKS-system models the 
evolution of cells (a phenomenon called {\it chemotaxis} in biology) interacting via a self-produced sensitivity agent
(called {\it chemoattractant}) and their natural habitat domain is the 
two-dimensional Euclidean space $\rt.$ Both the cells and the sensitivity agents are also subject to independent diffusive fluctuations. 
The $n$-component multi-species PKS-system  governed by the following system of equations:
\begin{equation}\label{kss}
\begin{cases}
& \partial_t \rho_i(x,t) = \Delta_x \rho_i(x,t) - \sj a_{ij}\nabla_x \cdot 
\left(\rho_i(x,t)\nabla_x u_j(x,t)\right), \ in \ \rt \times (0,\infty),\\
&-\Delta_x u_i(x,t) = \rho_i(x,t), \hspace{5.8 cm} \ \ in \  \rt \times (0,\infty), \\
&\  \rho_i(x,0) = \rho_i^0 , \ \ i=1,\ldots, n,
\end{cases}
\end{equation}
where $\rho_i(x,t)$ denotes the cell density of the $i$-th population, $ u_i(x,t)$ denotes the 
concentration of the {\it chemoattractant}, produced by the $i$-th population and $\rho_i^0$ is the initial cell distribution of the $i$-th population. The constants
$a_{ij}$ measures the sensitivity of the $i$-th population towards the chemical gradient produced
by the $j$-th population. 
If $a_{ij}>0$ (respectively, $a_{ij} <0$) then $i$-th population is attracted (respectively, repelled) by the $j$-th population
known as positive (respectively, negative) chemotaxis. In this article, we assume the sensitivity matrix $(a_{ij})$ is symmetric with non-negative entries 
$a_{ij}\geq 0, \ for \ all \ i,j$, termed  in the literature the {\it conflict free} case.

Since the solutions to the  
Poisson equation $-\Delta u = \rho$ is unique up to a harmonic function, we define
concentration of the chemoattractant $u_i$ by the Newtonian potential of $\rho_i$
\begin{align}\label{newtonian potential}
 u_i(x,t) = -\frac{1}{2\pi}\inte \ln|x-y|\rho_i(y,t) \ dy, \ \ \ i=1,\ldots,n.
\end{align}

There have been several prototype models for chemotaxis in the literature. The first of its kind has been proposed by Patlak \cite{Pa} in $1953$ and Keller and Segel \cite{KS} in $1970.$ The original model of \cite{KS} consists of a coupled parabolic-parabolic equation and comprehends the single species chemotaxis $n=1$. The parabolic-elliptic model is the quasi-equilibrium state of the chemoattractant
and when the time scale of observation is a lot smaller compared to the speed at which the chemoattractant degrades.
Over the last four decades PKS-system \eqref{kss} for $n=1 ( \ \mbox{and} \ a_{11} = a >0)$ has been widely studied in the literature, see \cite{Childress, NSenba, SeSu, SSbook, Subook, BT, BKLN, BD, BCM, BDP, BCCjfa, CD, FM} and the references therein.  

One of the main reasons for so many interests in the mathematical community is that the system adores a critical mass $\beta: = \inte \rho^0(x) \ dx$. In other words, the parameter $\beta$ solely determines the dichotomy between the global in time existence and the chemotactic collapse (or finite time blow up): if the initial number of bacteria is smaller than the critical threshold $\beta \leq 8\pi/ a$, then there exists a global in time solution \cite{BD, BCM, BKLN}. However, if it crosses the critical threshold, i.e. $\beta > 8\pi/ a$, then all the solutions blow-up in finite time. Thus completing the whole picture of the existence vs non-existence expedition. The above mentioned existence vs non-existence phenomena renders us to define the sub-critical regime i.e. $a \beta < 8\pi$, and the critical regime i.e. $a\beta = 8\pi$.
Moreover, in the critical case, if the second moment 
 of the initial data is finite, then the solutions do blow up
 in the form of a Dirac delta measure as time $t$ goes to infinity \cite{BCM}.

Chemotaxis for the multi-species population is quite prevalent, see for instance \cite{H} for biologigal motivations.
The congregation of different species of bacterias and chemicals interacting with each other in a habitat domain affect an individual species as well as the total population. Equation \eqref{kss} serves as an attempt to understand the underlying complex biological mechanism.
 The model \eqref{kss} has been proposed by the second author in \cite{W1} and subsequently further expanded in \cite{Horsurvey1, H}. The existence vs non-existence phenomena is also quite expected in the multi-species systems, because of the two opposing forces of equal order are competing against each other. The smoothing effect diffusion term $\Delta \rho_i$ and the weighted cumulative drift induced by the chemical gradients $\sj a_{ij}\nabla u_j$, which is assisting the cells to accumulate are competing.  However, the notion of sub-critical and critical mass is a lot more involved quantities (see Definition \ref{subcriticalcritical}). For two species model, Espejo et. al. \cite{two3, two5} found the curve in $\rt$ which provides fair analogue of critical mass as in the single species model. Moreover, the existence of global solutions have been obtained in the sub-critical case (i.e., masses which lie strictly below the curve). For masses lying strictly above the mentioned curve, finite time blow-up of solutions have also been observed (see Remark \ref{2species} below). However, the global existence for masses lying on the curve was left open.  We refer the readers to \cite{two1,two2,two3,two4,two5} for related works in $2$-species model. In \cite{KW19} we obtained global in time solutions in the sub-critical regime (see also \cite{HeTadmor}) for $n$-species. Moreover, if the mass crosses the critical zone, then the chemotactic collapse is inevitable. 

The main focus of this article is to study the global existence and large time asymptotic of solutions at the critical mass and for any $n$-number of populations.  The set of all critical mass contained in a $(n-1)$-dimensional ellipsoidal domain. Also, quite surprisingly we found that, at the critical mass regime, all the components of the solutions do concentrate in the form of a Dirac delta measure at the same point when the time approaches infinity.

\subsection{Mathematical analysis of PKS-system}
A solution $\bm{\rho} :=(\rho_1, \ldots, \rho_n)$ to the PKS-system \eqref{kss},
at least formally, satisfies the following identities: 
\begin{itemize}
 \item Conservation of mass:
 \begin{align} \label{mass Conservation}
   \inte \bm{\rho}(x,t) \ dx = \inte \bm{\rho}^0(x) \ dx = \bm{\beta}, \ \ for \ all \ t >0.
 \end{align}
 \item Conservation of the center of mass:
 \begin{align*} 
   \si\inte x\rho_i(x,t) \ dx = \si\inte x\rho^0_i(x) \ dx.
 \end{align*}
\item Free energy dissipation or the free energy identity:
\begin{align*}
 \f(\bm{\rho}(\cdot , t)) + \int_0^t \mathcal{D}_{\f}(\bm{\rho}(\cdot, s)) \ ds = \f(\bm{\rho}^0),
\end{align*}
\end{itemize}
where the free energy $\f$ is defined by
 \begin{align} \label{free energy}
 \f(\bm{\rho}) = \fl,
\end{align}
and the dissipation of free energy $\mathcal{D}_{\f}$ is defined by
\begin{align} \label{dissipation of f}
 \mathcal{D}_{\f}(\bm{\rho}) = \si  \inte \left|\frac{\nabla \rho_{i}(x)}{\rho_{i}(x)} 
 - \sj a_{ij}\nabla u_{j}(x)\right|^2 \rho_{i}(x) \ dx.
\end{align}
\noindent
$\bullet$ If the second moment of the initial condition $M_2(\bm{\rho}^0) :=\si \inte|x|^2\rho^0_i(x) \ dx$ 
is finite then formally 
\begin{align}\label{second moment} 
 M_2(\bm{\rho}(\cdot,t)) = \frac{\Lambda_{I}(\bm{\beta})}{2\pi}t + M_2(\bm{\rho}^0),
\end{align}
where $I = \{1,\ldots, n\}$ and $\Lambda_I(\bm{\beta})$ is a quadratic polynomial in $\bm{\beta}$ defined by
\begin{align}\label{lambda}
 \Lambda_J(\bm{\beta}) := \sum_{i \in J} \beta_i \left(8\pi - \sum_{j \in J} a_{ij}\beta_j\right), \ \ for \ all \ \emptyset \neq J 
 \subset I.
\end{align}

As a consequence of \eqref{second moment}, if $\Lambda_I(\bm{\beta}) < 0,$ then a solution can not exists globally. If $T^*$ is the maximal time of existence then necessarily 
$T^* \leq -\frac{2\pi M_2(\bm{\rho}^0)}{\Lambda_I(\bm{\beta})}.$ On the other hand, if
$\Lambda_I(\bm{\beta}) = 0,$ then the second moment is preserved throughout the time.  

\vspace{0.2 cm}

\noindent
$\bullet$ Moreover, We observe formally that the system \eqref{kss} can be written as
\begin{align} \label{formal gradient flow}
 \partial_t \rho_i = \nabla \cdot \left(\rho_i \nabla \frac{\delta \f}{\delta \rho_i}(\bm{\rho})\right), \ \ i=1,\ldots,n.
\end{align}

As a consequence, the dissipation of energy \eqref{dissipation of f} can be expressed by
\begin{align*}
 \mathcal{D}_{\f}(\bm{\rho}(\cdot,t))=\si\inte \Big|\nabla \frac{\delta\f(\bm{\rho})}{\delta \rho_i}\Big|^2 \rho_i(x,t) \ dx.
\end{align*}

Equation \eqref{formal gradient flow} is the formal structure of a gradient flow of the free energy $\f$
in the space $\mathcal{P}_2^{\beta_1}(\rt)\times \cdots \times \mathcal{P}_2^{\beta_n}(\rt)$ 
equipped with the 2-Wasserstein distance $\dw$ (see section \ref{notation section} for definition),
where $\mathcal{P}_2^{\beta_i}(\rt)$ denotes the space of non-negative Borel measures on $\rt$ with total mass 
$\beta_i$ and finite second moment and
$\frac{\delta \f}{\delta \rho_i}$ denotes the first variation of the functional $\f$ with respect to the 
variable $\rho_i.$
The functional $\f$ on the product space 
$\mathcal{P}_2^{\beta_1}(\rt)\times \cdots \times \mathcal{P}_2^{\beta_n}(\rt)$ is defined by $\f(\bm{\rho})$
if $\bm{\rho} \in \Gamma^{\bm{\beta}},$  where 
\begin{align*}
 \Gamma^{\bm{\beta}} = \left\{\bm{\rho}= (\rho_i)_{i=1}^n | \ \rho_i \in L^1_+(\rt),
\inte \rho_i(x)\ln \rho_i(x) \ dx < +\infty, 
\inte \rho_i(x)\ dx = \beta_i, \right. \\
 \left.  \ \inte \rho_i(x)\ln(1+|x|^2) \ dx<+\infty\right\}
\end{align*}
and $+\infty$ elsewhere.

 Before proceeding further let us first introduce the appropriate notion of a weak solution
 to the PKS-system \eqref{kss}. Throughout this article, we use the notation $\mathcal{H}(\bm{\rho})
 := \si \inte \rho_i \ln \rho_i$ to denote the entropy of the solutions and 
 $\mathcal{H}_+(\bm{\rho})
 := \si \inte \rho_i (\ln \rho_i)_+$ is the positive part of the entropy.
 \begin{Def} \label{weak sol}
  For any initial data $\bm{\rho}^0$ in $\Gamma^{\bm{\beta}}\cap \{\bm{\rho}  \ | \ M_2(\bm{\rho}) < \infty\}$ and $T^*>0$ we say that 
  a non-negative vector valued function $\bm{\rho} \in \left(C([0,T^*); \mathcal{D}^{\prime}(\rt))\right)^n$
  satisfying
  \begin{align*}
   \mathcal{A}_T(\bm{\rho}) :=\sup_{t \in [0,T]} \Big(\mathcal{H}_+(\bm{\rho}(t)) + M_2(\bm{\rho}(t))\Big)
+ \int_0^T \mathcal{D}_{\f}(\bm{\rho}(t)) \ dt < +\infty, \ \forall \ T \in (0,T^*)
  \end{align*}
is a weak solution to the PKS-system \eqref{kss} on the time interval $(0,T^*)$
associated to the initial condition $\bm{\rho}^0$ if $\bm{\rho}$ satisfies \eqref{mass Conservation}
and 
 \begin{align*}
  &\int_0^{T^*} \inte \partial_t\xi(x,t) \rho_i(x,t) \ dxdt + \inte \xi(x,0) \rho_i^0(x) \ dx \\
 &-\int_{0}^{T^*} \inte  \rho_i(x,t)\left( \frac{\nabla_x \rho_i(x,t)}{\rho_i(x,t)}  
 -\sj a_{ij}  \nabla_x u_j(x,t)\right) \cdot \nabla_x \xi(x,t) \ dxdt = 0
 \end{align*}
  for all $\xi \in C_c^{2}([0,T^*)\times \rt)$ and for all $i=1,\ldots, n.$ If $T^* = +\infty$ we say $\bm{\rho}$ is a global weak solution of the system.
  \end{Def}
Note that thanks to the finite dissipation assumption and Cauchy-Schwartz inequality, all the terms in the weak formulation makes sense.

For the convenience of the readers, let us write down our set of assumptions: 
\begin{Ass}\label{AssGamma}
The initial condition satisfies  $\bm{\rho}^{0} \in \ga$  where
\begin{align*}
 \ga:=\left\{ \bm{\rho} \in \Gamma^{\bm{\beta}} \ | \ M_2(\bm{\rho}) := \si \inte |x|^2\rho_i \ dx < +\infty, \ \si \inte x\rho_i(x) \ dx = 0 \right\},
\end{align*}
\end{Ass}
Note that 
the equation is translation invariant so the center of mass is conserved throughout the time. 
\begin{Ass}\label{AssLambda} 
We also assume
\begin{itemize}
 \item $A=(a_{ij})_{n \times n}$ symmetric and non-negative matrix satisfying $a_{ii}>0$ for all \\ 
$i \in I:=\{1,\ldots,n\}.$ 
 \item The initial mass $\bm{\beta}$ is critical i.e.,
 \begin{align} \label{beta condition}
 \Lambda_I(\bm{\beta}) = 0, \ \Lambda_J(\bm{\beta})> 0, \ \mbox{for all} \ \emptyset \neq J  \subsetneq I.
\end{align}
\end{itemize}
\end{Ass}

\subsection{Our approach and major difficulties}
There is an illuminating theory devoted to the gradient flows in Wasserstein-space in their book by Ambrosio, Gigli and Savar\'{e} \cite{AGS}. However, the functional $\f$ fails to satisfy the necessary convexity assumption in \cite{AGS} to have a complete well-posed theory. On the positive side, we can rely upon the PDE based approach of Wasserstein gradient flow. In precise, we could utilize De Giorgi's generalized minimizing movements \cite{DeGiorgi}, to study the PDE \eqref{kss}. 
Such connections first discovered by Otto \cite{Otto, Otto1} and subsequently, Jordan, Kinderlehrer and Otto \cite{JKO} implemented this idea for the class of Fokker-Plank equation and the heat equation. In the literature, 
this approach now referred to minimizing movement scheme or the JKO-scheme:
for a time step $\tau>0,$ we define recursively
\begin{align}\label{formal jko}
 \bm{\rho}^{k}_{\tau} \in \arg\min_{\bm{\rho}\in \ga} \left(\f(\bm{\rho}) +\frac{1}{2\tau} \dw^2
 (\bm{\rho}, \bm{\rho}^{k-1}_{\tau})\right), \ \ \ k \geq 1
\end{align}
with $\bm{\rho}^0_{\tau} = \bm{\rho}^0,$ provided all the minimizers exist. The goal is to show that an appropriate interpolation of the minimizers converge to a solution in the sense of Definition \ref{weak sol}.

The sharp conditions under which the functional involved are bounded below have been well studied in the literature. Indeed,
it follows from the results of \cite{CSW, SW} that
the bound from below of $\f$ and, in particular,  the functional $\bm{\rho}
\longmapsto \G:=\f(\bm{\rho}) + \frac{1}{2\tau}\dw^2(\bm{\rho}, \bm{\eta})$ depends on the 
the following relations of $\bm{\beta}$ and the interaction matrix $A$: 
\begin{align} \label{beta}
 \begin{cases}
  \Lambda_J(\bm{\beta}) \geq 0, \ \ for \ all \ \emptyset \neq J 
 \subset \{1,\ldots, n\}, \\
 \mbox{ {\it if  for  some}} \ J, \ \Lambda_J(\bm{\beta}) = 0, \ then \ a_{ii} + \Lambda_{J\backslash \{i\}}(\bm{\beta})>0, \ \forall 
 i \in J,
 \end{cases}
\end{align}
where $\Lambda_J({\bm{\beta}})$ is defined by \eqref{lambda}.
In particular, it is shown in \cite{CSW,SW} that 
$\Lambda_I(\bm{\beta}) = 0$ and \eqref{beta} is necessary and sufficient condition 
for the bound from below of $\f$ over $\ga.$ 
 Needless to say, if $a_{ii}>0$ for all 
$i \in I,$ then $\Lambda_I(\bm{\beta}) = 0$ and $\Lambda_J(\bm{\beta}) \geq 0$ for $J \neq I$ is necessary and sufficient condition 
for the bound from below for $\f$. In addition, in \cite{SW} the authors showed that $\f$ admits a minimizer in $\Gamma^{\bm{\beta}}$ if and only if \eqref{beta condition}
is satisfied, which allows us to define the notion of critical mass.

\begin{Def} \label{subcriticalcritical}
 Given a symmetric non-negative matrix $A.$  
 \begin{itemize}
  \item $\bm{\beta}$ is said to be sub-critical if
 \begin{align*}
  \Lambda_J(\bm{\beta}) > 0, \ \ \ for \ all \ \emptyset \neq J \subset I.
 \end{align*}
 \item $\bm{\beta}$ is said to be critical if
 \begin{align*}
  \Lambda_I(\bm{\beta}) = 0, \ and \ \Lambda_J(\bm{\beta}) > 0, \ \ \ for \ all \ \emptyset \neq J \subsetneq I.
 \end{align*}

 \end{itemize}
\end{Def}

The criterion for existence of minimizers in \eqref{formal jko} drastically differs from that of $\f.$ Indeed, for sub-critical mass $\bm{\beta}$ there always exists a minimizer in \eqref{formal jko}. And the global existence of solutions to \eqref{kss} has also been dealt with in \cite{KW19}. In the sub-critical case, the sharp condition on $\bm{\beta}$ for the bound from below of $\f$ gives us the uniform entropy bound on any minimizing sequence (the readers can consult \cite{jko2, KW}). However, the arguments of the sub-critical regime do not work in the critical case, and apparently, one can not rule out the possibilities of concentration of minimizing sequences.

There are two major difficulties: first of all, it is not clear that the MM-scheme \eqref{formal jko} is well defined for critical $\bm{\beta}$. In section \ref{section critical}, we study this delicate point and obtain a sufficient criterion for the existence of minimizers. Indeed, we derive a necessary and sufficient criterion (Theorem \ref{mms existance thm}) for which any minimizing sequence satisfy uniform entropy bound. In particular, we show that, given an initial datum $\bm{\rho}^0,$ there exists a $\tau^* \in (0,1)$ such that for every $\tau \in (0,\tau^*),$ the MM-scheme is well defined. 

The next hurdle is to obtain uniform estimates on the minimizers obtained in \eqref{formal jko}. In particular, obtaining uniform entropy estimates is  dealt with in section \ref{section apriori}. 
This  follows from a slight modification of the proof of existence theorem established in section \ref{section critical} (see Remark \ref{mms existence remark}) together with a gain of integrability result Lemma \ref{better integrability}. 

Using the Euler-Lagrange equation (Theorem \ref{el}), we obtain a discrete version of the second moment conservation identity (Lemma \ref{wdestimate}). As a consequence, if the interpolates concentrates at the origin,
in the limit we would get $M_2(\bm{\rho}^0) = 0,$ contradicting $\bm{\rho}^0 \in \ga.$  In passing to the limit in the discrete second moment identity we must obtain uniform integrability on the family $M_2(\rk), k \tau \leq T.$ A very classical result due to de la Vell\'{e}e Poussin's (see Lemma \ref{de la lemma} below) states that a family $\mathcal{Z}$ of $L^1$ measurable functions with respect to a measure $\mu$ is uniformly integrable if and only if there exists a convex super-linear function at infinity $\vu$ satisfying $\sup_{g \in \mathcal{Z}}\int \vu(g) d\mu < \infty.$ In precise, we need to find a convex function $\vu$ satisfying certain growth assumptions such that the $L^1$-norm of $\vu(|x|^2)$ with respect to $\rk$ remains uniformly bounded. 
 This estimate  implies the uniform bound on higher moments in terms of the bound on the initial datum (see Lemma \ref{better integrability} below). 

\subsection{Main results} The main results of this article are as follows:

\begin{Th}\label{main}
Assume $\bm{\rho}^0 \in \ga$ and $\bm{\beta}$ is critical. Further assume that the interaction matrix $(a_{ij})$ have 
strictly positive diagonal entries. Then the PKS-system \eqref{kss} admits 
a global weak solution $\bm{\rho}$ in the sense of Definition \ref{weak sol} with initial data $\bm{\rho}^0.$
Moreover, $\bm{\rho}$ satisfies for every $T>0$
\begin{itemize}
  \item[(a)] $\bm{\rho} \in \left(L^2((0,T)\times \rt)\right)^n \cap \left(L^1(0,T;W^{1,1}(\rt))\right)^n,$ and 
  Fischer information bound:   
 \begin{align*}
 \si\int_{0}^{T}  \inte \left|\frac{\nabla \rho_{i}(x,t)}{\rho_{i}(x,t)}\right|^2 \rho_{i}(x,t) \ dxdt < +\infty.
\end{align*}
\item[(b)] Free energy inequality:
\begin{align*} 
\si\int_{0}^{T}  \inte \left|\frac{\nabla \rho_{i}(x,t)}{\rho_{i}(x,t)} 
 - \sj a_{ij}\nabla u_{j}(x,t)\right|^2 \rho_{i}(x,t) \ dxdt + \f(\bm{\rho}(T)) \leq \f(\bm{\rho}^0).
\end{align*}
\end{itemize}
\end{Th}

\begin{Rem} \label{2species}
The following two species model has been studied in \cite{two3, two5}:
\begin{align} \label{studied}
\begin{cases}
\partial_t \tilde \rho_1(x,t) = \Delta_x \tilde \rho_1(x,t) - \chi_1 \nabla_x \cdot (\tilde \rho_1(x,t) \nabla_x \tilde u(x,t))\\
\partial_t \tilde \rho_2(x,t) = \Delta_x \tilde \rho_1(x,t) - \chi_2 \nabla_x \cdot (\tilde \rho_2(x,t) \nabla_x \tilde u(x,t))\\
\Delta_x \tilde u(x,t) = \tilde \rho_1(x,t) + \tilde \rho_2(x,t),
\end{cases}
\end{align}
in $\rt \times (0,\infty).$ It has been shown \cite{two3,two5} that provided the masses $\tilde \beta_i = \inte \tilde \rho_i$ satisfy
\begin{align}
\frac{8\pi}{\chi_1} > \tilde \beta_1, \ \frac{8\pi}{\chi_2}> \tilde \beta_2, \ and \ 8\pi\left(\frac{\tilde \beta_1}{\chi_1} + \frac{\tilde \beta_2}{\chi_2}\right) - (\tilde \beta_1 + \tilde \beta_2)^2 >0,
\end{align}
there exists a global in time solution. Moreover, it was observed that if one of the inequality $>$ is replaced by 
$<$ then the global existence may fail. 

Note that for a solution $(\rho_1, \rho_2)$ to \eqref{kss} corresponding to the interaction matrix $a_{11} = \chi_1^2,a_{12} = a_{21} = \chi_1\chi_2, a_{22} = \chi_2^2,$ if we define $\tilde \rho_i = \chi_i \rho_i, \tilde u = \chi_1 u_1 + \chi_2 u_2$ then $(\tilde \rho_1, \tilde \rho_2, \tilde u)$ solves \eqref{studied} and our criterion (Assumption \ref{AssLambda}) for global existence translates into $\frac{8\pi}{\chi_1} > \tilde \beta_1, \ \frac{8\pi}{\chi_2}> \tilde \beta_2$ and
\begin{align} \label{comparison_2}
 8\pi\left(\frac{\tilde \beta_1}{\chi_1} + \frac{\tilde \beta_2}{\chi_2}\right) - (\tilde \beta_1 + \tilde \beta_2)^2 \geq 0.
\end{align}
In particular, if $\frac{\chi_1}{\chi_2} \in (\frac{1}{2}, 2)$ then the lines $\tilde \beta_i = \frac{8\pi}{\chi_i}$ does not intersect the curve \eqref{comparison_2} and hence \eqref{comparison_2} is necessary and sufficient criterion for the global existence of solutions to \eqref{studied}. In this special case, the curve \eqref{comparison_2} truly represents the analogue of critical mass as in the single species model.  
\end{Rem}

A solution to \eqref{kss} satisfying $(a)$ and $(b)$ of Theorem \ref{main} is called free energy solution. 
By now it is well known that free energy solutions are of class $C^{\infty}$ in both space and 
in time away from $t=0.$ Moreover, such solutions found to be unique. The proof of smoothness and uniqueness relies on the a posteriori estimates 
and depends on the novel ideas of DiPerna and Lions renormalized solutions. We refer the interested readers to 
\cite{FM} for more details, where the authors proved smoothness and uniqueness for single species population (see \cite{HeTadmor} for 
multi-species counterpart).

Our second main result deals with the asymptotic behaviour of the solutions as time $t \rightarrow \infty$.

\begin{Th}\label{main2}
  Assume $A$ and $\bm{\beta}$ satisfy the assumptions of Theorem \ref{main} and $\bm{\rho}^0 \in \ga$. Then for any free energy solution $\bm{\rho}$ to \eqref{kss}  
 \begin{align*}
  \lim_{t \rightarrow \infty} \bm{\rho}(\cdot, t) = \bm{\beta} \delta_{\bm{0}},
 \end{align*}
in the weak* convergence of measures.
\end{Th}

The organization of this article are as follows: in section \ref{notation section}, we recalled necessary definitions and known results, that are essential in this article. The topic of section \ref{section mms} focuses on the description of the MM-scheme and their regularity results. Section \ref{section el} devoted to the Euler-Lagrange equations for the minimizers obtained through MM-scheme, and as a consequence, we derived several moment estimates. We proved, in section \ref{section critical}, the necessary and sufficient criterion for the boundedness of every minimizing sequences for the functionals $\G$ at the critical mass regime, which ensures the existence of a minimizer and the well-definedness of the MM-scheme. In section \ref{section apriori}, we established the a priori estimates of the minimizers obtained by MM-scheme. Section \ref{section global existence} devoted to the proof of convergence of the scheme, providing the existence of a global weak solution satisfying the free energy inequality. Finally, in section \ref{section asymptotic}, we proved the aggregation to the Dirac measure as time $t$ approaches to infinity. At the end, we have an appendix section where we give a proof of Lemma \ref{better integrability} and recalled a well known compactness of vector fields lemma.

\section{Notations and Preliminaries}\label{notation section}
In this section we have listed the main notations used in this article and also some of the well known results about the Wasserstein distance and the free 
energy functional $\f.$

Any bold letters will be used to denote $n$-vectors or $n$-vector valued functions. For example $\bm{\rho} = (\rho_1, \ldots, \rho_n) \in (L^1_+(\rt))^n, 
\ \bm{\beta} = (\beta_1, \ldots, \beta_n) \in (\mathbb{R}_+)^n$
and so on. The entropy of a scalar function $\rho \in L^1_+(\rt)$ will be denoted by $\mathcal{H}(\rho) := \inte \rho(x)\ln \rho(x) \ dx.$ For the vector valued functions 
$\bm{\rho} \in (L^1_+(\rt))^n$ the entropy is also denoted by $\mathcal{H}(\bm{\rho})$ and is defined by 
\begin{align*}
 \mathcal{H}(\bm{\rho}) = \si \mathcal{H}(\rho_i).
\end{align*}
We will use similar definitions for the second moment: 
\begin{align*}
M_2(\bm{\rho}) = \si M_2(\rho_i) := \si \inte |x|^2\rho_i(x) \ dx. 
\end{align*}

Let $\mathcal{P}(\rt)$ be the space of all Borel probability measures on $\rt,$ 
$\mathcal{P}_{2}(\rt)$  denotes the subset of $\mathcal{P}(\rt)$ having finite second
moments and $\mathcal{P}_{ac,2}(\rt)$ denotes the subset of $\mathcal{P}_2(\rt)$ which are
absolutely continuous with respect to the Lebesgue measure on $\rt.$
 
Given two elements $\mu, \nu$ of $\mathcal{P}(\rt)$ and a Borel map $T: \rt \rightarrow \rt,$ we say $T$ pushes
forward $\mu$ to $\nu,$ denoted by $T\#\mu = \nu,$ if for every Borel measurable subset $U$ of $\rt,$
$\nu(U) = \mu(T^{-1}(U)).$ Equivalently, 
\begin{align}\label{cov}
 \inte \psi(x) d\nu(x) = \inte \psi(T(x)) d\mu(x), \ \ for \ every \ \psi \in L^1(\rt,d\nu).
\end{align}

\subsection{Wasserstein distance}
Our main sources for the optimal transportation problems and Wasserstein distances are the  first book of Villani \cite{Villani03}, the comprehensive lecture notes on the subject 
of gradient flows in Wasserstein spaces written by Ambrosio, Gigli and Savar\'{e} \cite{AGS} and the recent book by Santambrogio \cite{San}. The results we are going to 
recall in this subsection are very classical and can be found in any of these text books.
\subsubsection{2-Wasserstein distance}
On $\mathcal{P}_2(\rt)$ we can define a distance $\dws,$ using the Monge-Kantorovich transportation problem with quadratic
cost function $c(x,y) = |x-y|^2.$ More precisely, given $\mu, \nu \in \mathcal{P}_2(\rt)$ define
\begin{align}\label{mk}
 \dws^2(\mu,\nu) := \inf_{\pi \in \Pi(\mu,\nu)} \left[\inte \inte |x-y|^2 \ d\pi(x,y) \right],
\end{align}
where
\begin{align*}
\Pi(\mu,\nu):= \{\pi \in \mathcal{P}(\rt \times \rt)| \ (P_1) \# \pi= \mu, \ (P_2) \# \pi= \nu\},
\end{align*}
is the set of transport plans and $P_i: \rt \times \rt \rightarrow \rt$ denotes the canonical projections 
on the $i$-th factor.

A well known theorem of Brenier \cite{Brenier} asserts that: if $\mu \in \mathcal{P}_{ac,2}(\rt)$ then there exists a unique 
(up to additive constants) convex, lower semi continuous function $\varphi$ such that $\nabla \varphi \# \mu = \nu$ and the optimal
transference plan $\hat \pi$ on the right hand side of \eqref{mk} is given by $\hat \pi = (I, \nabla \varphi)
\# \mu,$ where $I : \rt \rightarrow \rt$ is the identity mapping (see \cite{McCann2}, \cite[Theorem $2.12$]{Villani03}). 
As a consequence, we have 
\begin{align} \label{wd1}
\dws^2(\mu, \nu) = \inte |x-\nabla \varphi(x)|^2 d\mu(x), \ \ \ where \ 
 \nabla \varphi \# \mu = \nu.
\end{align}

If $\mu, \nu$ are two non-negative measures on $\rt$ (not necessarily probability measures) satisfying 
the total mass compatibility condition
$\mu(\rt) = \nu(\rt) (=\beta >0),$ then we define the $2$-Wasserstein distance between them as
follows:
\begin{align}\label{wd}
 \dws(\mu,\nu) = \beta^{\frac{1}{2}}\dws\left(\frac{\mu}{\beta}, \frac{\nu}{\beta}\right).
\end{align}
We will denote by $\mathcal{P}^{\beta}(\rt)$ the space of non-negative Borel measures with total mass $\beta$
and $\mathcal{P}^{\beta}_{2}(\rt)$ and $\mathcal{P}^{\beta}_{ac,2}(\rt)$ are defined analogously. We will also
use the bold $\bm{\beta}$ notation in $\mathcal{P}^{\bm{\beta}}(\rt)$ to denote the product space $\mathcal{P}^{\beta_1}(\rt)
\times \cdots \times \mathcal{P}^{\beta_n}(\rt).$

One advantage of defining the Wasserstein distance on $\mathcal{P}^{\beta}(\rt)$ by \eqref{wd} is that, 
if $\mu$ is absolutely continuous with respect 
to the Lebesgue measure and if $\nabla \varphi$ is the gradient of a convex function pushing $\mu/\beta$ forward to 
$\nu/\beta$ then $\nabla \varphi \# \mu = \nu$ and 
\begin{align*}
 \dws^2(\mu, \nu) = \inte |x - \nabla \varphi (x)|^2 d\mu(x),
\end{align*}
where note that $\dws(\mu, \nu)$ is defined by \eqref{wd}.

\subsubsection{Kantorovich-Rubinstein distance}
If the cost function $c(x,y)$ is the Euclidean distance $|x-y|,$ then 
\begin{align}\label{w1distance}
 \dwso(\mu, \nu) = \beta^{\frac{1}{2}}\dwso\left(\frac{\mu}{\beta}, \frac{\nu}{\beta}\right)=\beta^{\frac{1}{2}}\inf_{\pi \in \Pi(\frac{\mu}{\beta}, \frac{\nu}{\beta})} \inte \inte |x-y| \ \pi(dxdy) ,
\end{align}
also defines a distance on $\mathcal{P}_2^{\beta}(\rt)$ and is called the $1$-Wasserstein distance. 
The constant $\beta^{\frac{1}{2}}$ on the right hand side of \eqref{w1distance} does not really matter for 
this article. We defined it  in order to be consistent with our definition of $2$-Wasserstein distance on $\mathcal{P}_2^{\beta}(\rt).$ On unbounded domains, such as 
our case $\rt,$ $\dwso$ is weaker than $\dws$ in the sense that $\dwso( \cdot, \cdot) \leq \dws( \cdot, \cdot).$ There is however a sharp distinctive feature of 
the $1$-Wasserstein distance, it is the pinning property: $\dwso(\mu, \nu)$ depends only on the difference $\mu-\nu.$ This is manifested by the alternative dual 
formulation of $\dwso:$
\begin{align} \label{w1dual}
 \dwso(\mu, \nu) = \beta^{-\frac{1}{2}}\sup\left\{\inte\phi(x)d(\mu - \nu)(x) | \ ||\phi||_{Lip}\leq 1\right\}
\end{align}
where $||\phi||_{Lip} := \sup_{x \neq y}(\phi(x) - \phi(y))/|x-y|.$ Under fairly reasonable assumptions on $\mu, \nu,$ such as both $\mu$ and $\nu$
are absolutely continuous with respect to the Lebesgue measure, one can impose additional assumption that $\phi \in C_c^1$ in the supremum \eqref{w1dual}.

We define the Wasserstein distance for vector valued functions $\bm{\rho}, \bm{\eta} \in \mathcal{P}^{\bm{\beta}}(\rt)$ as follows:
\begin{align*}
 \dw(\bm{\rho}, \bm{\eta})  := \left[\si\dws^2 (\rho_i, \eta_i)\right]^{\frac{1}{2}}; \ \ \ \ 
 \dwo(\bm{\rho}, \bm{\eta}) := \si \dwso(\rho_i, \eta_i).
\end{align*}

We end this subsection with a result regarding the weak* lower semi-continuity  of the $2$-Wasserstein distance: a sequence of measures $\mu^m \in 
\mathcal{P}^{\beta}(\rt)$ is said to converge to $\mu \in \mathcal{P}^{\beta}(\rt)$ in the weak* topology of measures if
\begin{align*}
 \lim_{m \rightarrow \infty}\inte g(x) \ d\mu^m(x) = \inte g(x) \ d\mu(x)
\end{align*}
for every $g \in C_b(\rt),$ where $C_b(\rt)$ denotes the space of all bounded continuous functions in $\rt.$
\begin{Lem}[Weak* lower semi-continuity of $\dws$]\label{semicontinuity of dw}
 Let $\{\mu^m\}, \{\nu^m\}$ be a sequence in $\mathcal{P}^{\beta}_2(\rt)$ converging to $\mu, \nu \in \mathcal{P}^{\beta}_2(\rt),$ respectively, in the weak* topology 
 of measures. Further assume that the second moments $M_2(\mu^m), M_2(\nu^m)$ are uniformly bounded by some constant $C<\infty$ independent of $m.$ Then 
 \begin{align*}
  \dws(\mu, \nu) \leq \liminf_{m \rightarrow \infty} \dws(\mu^m, \nu^m).
 \end{align*}

\end{Lem}

 \subsection{Properties of the free energy functional}
 We recall a few  properties of the free energy functional $\f$
  whose proof can be found in \cite{SW,KW}.
 \begin{Prop}\label{propertyf}
 The followings hold:
  \begin{itemize}
   \item[(a)] $\f$ is bounded from below on $\ga$ if and only if $\bm{\beta}$ satisfies 
   \begin{align*}
    \Lambda_I(\bm{\beta}) = 0 \ \ and \ \eqref{beta}.
 \end{align*}
 \item[(b)] For any $n$-numbers $\alpha_i >0,$ the functional 
 \begin{align*}
 \f_{\bm{\alpha}}(\bm{\rho}):=\f(\bm{\rho}) + \si \alpha_iM_2(\rho_i) 
 \end{align*}
 is bounded from below on $\ga$ if and only if $\bm{\beta}$ satisfies \eqref{beta}.
 \item[(c)] The functionals $\f$ and $\f_{\bm{\alpha}}$ are sequentially lower semi-continuous with respect to the
 weak topology of $L^1(\rt).$
 \item[(d)] If $\bm{\beta}$ if sub-critical and $\min_{i \in I} \alpha_i>0,$ then all the sub-level 
 sets $\{\f_{\bm{\alpha}} \leq C\}$ are sequentially precompact 
 with respect to the weak topology of $L^1(\rt).$
 \end{itemize}
 \end{Prop}
In this sequel we will be using a slight variant of the functional $\f$ and the sharp conditions for bound from below as stated in Proposition \ref{propertyf}(a).
Naturally, it follows form a simple scaling argument. 
Since we assumed $a_{ii} >0$ for all $i \in I,$ let us only 
state a weaker version of it, which is enough for our purpose. Let $A=(a_{ij})$ be as before satisfying $a_{ii}>0$ for all $i \in I$, then for any  $n$-numbers $b_i >0$
the functional 
\begin{equation} \label{new functional}
 \si b_i \inte \rho_i \ln \rho_i + \si \sj \frac{a_{ij}}{4\pi}\inte \inte \rho_i(x) \ln |x-y| \rho_j(y)
\end{equation}
is bounded from below over $\ga$ if and only if $\Lambda_I(\bm{\beta}; \bm{b}) = 0$ and $\Lambda_J(\bm{\beta}; \bm{b}) \geq 0$ for all $\emptyset \neq J \subsetneq I,$
where $\Lambda_J(\bm{\beta}; \bm{b})$ is defined by
\begin{align}\label{lambda  b beta}
\Lambda_J(\bm{\beta}; \bm{b}) := 8\pi\sum_{i \in J} b_i\beta_i  - \sum_{i \in J}\sum_{j \in J} a_{ij}\beta_i\beta_j, \ \ for \ all \ \emptyset \neq J \subset I.
\end{align}

\begin{Rem} \label{remark lhls}
 The same conclusions of Proposition \ref{propertyf} hold true for the functional \eqref{new functional} as well, of course, one needs to 
 frame the conditions in terms of $\Lambda_J(\bm{\beta}; \bm{b}).$ The smallest lower bound of the functional $\f$ or \eqref{new functional}, whenever it is finite,
 will always be denoted by $C_{\mbox{\tiny{LHLS}}}(\bm{\beta}).$ 
\end{Rem}

\begin{Prop} \label{gnu proposition}
\begin{itemize}
\item[(a)] Assume $\bm{\beta}$ satisfies \eqref{beta} and fix $\bm{\eta} \in \ga$ and $\tau>0.$ 
 Then the functional 
 $\mathcal{G}_{\bm{\eta}}: \ga \rightarrow \mathbb{R}$ defined by
 \begin{align*}
  \mathcal{G}_{\bm{\eta}}(\bm{\rho}) := \f(\bm{\rho}) + \frac{1}{2\tau}\dw^2(\bm{\rho},\bm{\eta}) 
 \end{align*}
is bounded from below on $\ga.$ Moreover, $\mathcal{G}_{\bm{\eta}}$ is sequentially lower semicontinuous with respect to the 
weak topology of $L^1(\rt).$
\item[(b)] If $\bm{\beta}$ is sub-critical then all the sub-level sets $\{\mathcal{G}_{\bm{\eta}} \leq C\}$ are 
sequentially precompact 
with respect to the weak topology of $L^1(\rt).$ In particular, the minimization problem $\inf_{\bm{\rho \in \ga}}
\mathcal{G}_{\bm{\eta}}(\bm{\rho})$ admits a solution.

\item[(c)] For any $\bm{\beta}$ satisfying \eqref{beta}, let $\{\bm{\rho}^m\}$ be a minimizing sequence for $\inf_{\bm{\rho \in \ga}} \G.$ If the entropy $\mathcal{H}(\bm{\rho}^m)$
and the second moment $M_2(\bm{\rho}^m)$ are uniformly bounded above by some constant $C,$ independent of $m,$ then up to a subsequence $\bm{\rho}^{m}$ converges,
in the weak topology of $L^1(\rt)$, to a minimizer $\bm{\tilde \rho} \in \ga.$
\end{itemize}
\end{Prop}

\begin{Rem}
 The proof of Proposition \ref{gnu proposition}(a) (b) follows from Proposition \ref{propertyf}, the weak* lower semi-continuity of $\dw$ and the inequality 
 \begin{align} \label{1}
   M_2(\bm{\rho}) \leq 2 \dw^2(\bm{\rho}, \bm{\eta}) + 2M_2(\bm{\eta}).
  \end{align}
The conclusion $(c)$ is just a restatement of $(a)$ phrased in a different way and follows form Dunford-Pettis theorem. From the entropy and second moment bound it follows that the measures $\bm{\rho}^m$ are tight 
and converges to $\bm{\tilde \rho}$ in the weak* convergence of measures having a density with respect to the Lebesgue measure. This is enough to improve the weak*
convergence of measures to the
weak convergence in $L^1(\rt).$ We refer to \cite{KW} for a proof in this framework.
\end{Rem}

 \section{Minimizing Movement Scheme}\label{section mms}
Given two elements $\bm{\rho}, \bm{\eta} \in \mathcal{P}^{\bm{\beta}}_2(\rt),$ recall that the Wasserstein distance between them is defined by
\begin{align*}
 \dw(\bm{\rho}, \bm{\eta}) 
 = \left[\si \dws^2(\rho_i, \eta_i)\right]^{\frac{1}{2}},
\end{align*}
where $\dws$ is defined as in \eqref{wd}. Further recall that we are interested in the case $\bm{\beta}$-critical.
\subsection{Minimizing Movement Scheme (MM-scheme):} Given $\bm{\rho}^0 \in \ga$ and a time step $\tau
\in (0,1)$ sufficiently small, we set $\bm{\rho}^0_{\tau} = \bm{\rho}^0$ and define recursively
\begin{align}\label{mms}
 \bm{\rho}^{k}_{\tau} \in \arg \min_{\bm{\rho}\in \ga} \left\{\f(\bm{\rho}) + \frac{1}{2\tau}\dw^2(\bm{\rho},
 \bm{\rho}^{k-1}_{\tau})\right\}, \ \ k \in \mathbb{N}.
\end{align}
 It is not yet clear that the scheme is well defined. We will prove in section \ref{section critical} (Theorem \ref{mms existance thm}) a sufficient condition for the existence of minimizers. In particular, if the initial datum $\bm{\rho}^0 = \bm{\rho}^0_{\tau}$ satisfies  
\begin{align}\label{r0}
 \f(\bm{\rho}^0_{\tau}) < \inf_{\bm{\rho} \in \ga} \F + \frac{1}{2\tau} M_2(\bm{\rho}^0_{\tau}),
\end{align}
then there exists a solution $\bm{\rho}^1_{\tau}$ to the minimization problem $\min_{\bm{\rho}\in \ga} (\f(\bm{\rho}) + \frac{1}{2\tau}\dw^2(\bm{\rho},
 \bm{\rho}^0_{\tau}))$ and moreover, $\bm{\rho}^1_{\tau}$ satisfies
 \begin{align}\label{r1}
   \f(\bm{\rho}^1_{\tau}) < \inf_{\bm{\rho} \in \ga} \F + \frac{1}{2\tau} M_2(\bm{\rho}^1_{\tau}).
 \end{align}
Evidently, \eqref{r0} is satisfied if $\tau$ is small enough.  For the rest of this article we will always assume $\tau \in (0,\tau^*),$ 
where $\tau^*$ is small enough such that the initial condition $\bm{\rho}^0$ satisfies \eqref{r0} with $\tau$ replaced by $\tau^*$.

Similarly, as a consequence of \eqref{r1} and Theorem \ref{mms existance thm}, $\bm{\rho}^2_{\tau}$ exists and satisfies \eqref{r1} with $\bm{\rho}^1_{\tau}$ 
replaced by $\bm{\rho}^2_{\tau}.$
By induction, we conclude that the minimizing movement scheme \eqref{mms} is well defined for every $\tau \in (0, \tau^*).$

In section \ref{section global existence}, we will prove that a suitable interpolation of these minimizers converge to a weak solution to \eqref{kss}. As mentioned in the introduction that there are two major difficulties to obtain global existence: (a) well definedness of the $MM$-scheme and (b) uniform estimates on the entropy and the second moment. If we carefully trace back  the existing literature on the PKS-system 
(in particular, the gradient flow approach), one can easily understand that the remaining analysis leading to the global existence depends only on these two issues.

The first one of these results (Lemma \ref{regularity} below) deals with the regularity of the minimizers. These regularity results have been proved rigorously in our earlier article 
\cite{KW19}, Lemma $4.3$],   which relies on the flow interchange technique introduced by Matthes-McCann and Savar\'{e} \cite{MMS}.

\subsection{Regularity of the minimizers}
\begin{Lem}\label{regularity}
Let $\tau \in (0,\tau^*)$ and let $\rk$ be a sequence obtained using the MM-scheme \eqref{mms} 
satisfying
\begin{align}\label{theta}
 \si\inte \rki |\ln \rki| \ dx + M_2(\rk) \leq \Theta,
\end{align}
for some constant $\Theta>0.$ Then $\rki \in W^{1,1}(\rt), \frac{\nabla \rki}{\rki} \in L^2(\rt,\rki)$ for all 
$i \in I.$

Moreover, there exists a constant $C(\Theta)$ such that 
\begin{align} \label{w11 estimate}
 \si \inte \left|\frac{\nabla \rki(x)}{\rki(x)}\right|^2 \rki(x) \ dx \leq \frac{2}{\tau}
 \left[\mathcal{H}(\rka) - \mathcal{H}(\rk)\right] + C(\Theta).
\end{align}
\end{Lem}
These regularity results, in particular, the $W^{1,1}$-regularity is necessary in persuasion of the Euler-Lagrange equations (see section \ref{section el}) satisfied by the minimizers. Interested readers can also consult \cite{Lestimate, Lestimate1} for related regularity results in this context.

 \section{The Euler-Lagrange equations and moment estimates} \label{section el}
 
 In this section we will derive several auxiliary results, assuming that a minimizer exists. Our main focus will be on the situation of critical $\bm{\beta}.$ However, most 
 of these results hold true irrespective of $\bm{\beta},$ as long as a minimizer exists. We begin by recalling the Euler-Lagrange equation satisfied by a minimizer.

\begin{Lem}\label{el}
 Let $\et \in \ga$ and $\tau>0$ be given and assume that $\vr$ be a minimizer of  $\inf_{\bm{\rho} \in \ga} \G.$
 Let $\nabla \varphi_i$ denotes the map transporting $\eta_i$ to $\vri$ then
 \begin{itemize}
  \item[(a)] The Euler-Lagrange equation: 
 \begin{align}\label{el1}
  \frac{1}{\tau}\inte (\nabla \varphi_i(x)-x) \cdot  \zeta(\nabla\varphi_i(x))\eta_i(x)dx
   =& -\inte  \zeta(x) \cdot\nabla\vri(x)dx \notag\\ 
   &+ \sj a_{ij} \inte  \zeta(x) \cdot \nabla u_j(x)\vri(x)dx,
 \end{align}
holds for all $i=1,\ldots, n$ and any $\zeta \in C_c^{\infty}(\rt;\rt),$ where
\begin{align*}
 u_j(x):= -\frac{1}{2\pi}\inte \ln|x-y| \vri(y)\ dy.
\end{align*}

\item[(b)] Free energy production term:
the following identity holds
\begin{align}\label{el2}
 \frac{1}{\tau^2}\dws^2(\vri, \eta_i) = 
 \inte \Big|\frac{\nabla \vri(x)}{\vri(x)} - \sj a_{ij}\nabla u_j(x)\Big|^2 \vri(x) dx.
\end{align}

\item[(c)] The consequent approximate weak solution is satisfied
\begin{align} \label{elfinal}
 &\left|\inte \psi(x)\left(\vri(x) - \eta_i(x)\right) + \tau\inte  \nabla\psi(x) \cdot \nabla \vri \right. \notag\\
 &\left. -\tau\sj a_{ij} \inte  \nabla \psi(x) \cdot \nabla u_j(x)\vri(x)dx \right| 
 = O\left(||D^2\psi||_{L^{\infty}}\right) \dws^2(\vri, \eta_i),
\end{align}
for all $\psi \in C_c^{\infty}(\rt).$

\item[(d)] In particular, for all $\psi \in C_c^{\infty}(\rt)$
\begin{align*}
  \frac{1}{\tau}&\si\inte (\nabla \varphi_i(x)-x) \cdot  \nabla\psi(\nabla\varphi_i(x))\eta_i(x)dx
   = \si\inte  \Delta \psi(x) \vri(x)dx \notag\\ 
   &- \si\sj \frac{a_{ij}}{4\pi} \inte \inte  \frac{(\nabla \psi(x) - \nabla \psi(y))\cdot (x-y)}{|x-y|^2}\vri(x)\varrho_j(y)dx.
\end{align*}

\end{itemize}
\end{Lem}
A proof of Lemma \ref{el} can be found in \cite[Lemma $5.1$]{KW19}. The proof is obtained by pursuing the main ideas of the seminal work by Jordan-Kinderlehrer and Otto \cite{JKO} 
with necessary modifications. See also \cite[Theorem $3.4$]{jko2} for the Euler-Lagrange equation related to the PKS-system of single population. Note that by the regularity results of 
Lemma \ref{regularity}, $\vri \in W^{1,1}(\rt)$ and hence all the terms in \eqref{el1} makes sense. The conclusion $(d)$ is a particular case of $(a)$ and follows by plugging 
$\zeta = \nabla \psi$ and $\nabla u_i (x)= -\frac{1}{2\pi}\inte \frac{x-y}{|x-y|^2} \vri(y) \ dy$ in \eqref{el1} and summing over all $i=1,\ldots,n.$ 

\subsection{Consequences of the Euler Lagrange equation} 
We introduce the following cut off function: \ Let $\Psi$ be a non-negative smooth function such that $\Psi \equiv 1$ in $B(0,1)$ and vanishes outside $B(0,2).$ We set 
$\Psi_R(x):= \Psi(\frac{x}{R})$ where $R>0.$ Then
\begin{align}\label{cut off}
 |\nabla \Psi_R(x)| = O\left(\frac{1}{R}\right),  \ |D^2 \Psi_R(x)| = O\left(\frac{1}{R^2}\right)
\end{align}
uniformly in $x\in\R^2$.

\subsubsection{Center of mass preservation and discrete second moment identity}
\begin{Lem} \label{zero com}
  Let $\et \in \ga$ and $\tau>0$ be given and assume that $\vr$ be a minimizer of  $\inf_{\bm{\rho} \in \ga} \G.$ Then 
  \begin{align*}
   \si \inte x \vri(x) \ dx = \si \inte x\eta_i(x) \ dx = \bm{0}.
  \end{align*}
\end{Lem}

\begin{proof}
Fix $v \in \rt.$ We will approximate the linear functional $x \cdot v$ by smooth compactly supported functions. We use $\psi_R(x):= (x\cdot v) \Psi_R(x)$
as a test function in the Euler-Lagrange equation Lemma \ref{el1}(d). Let us estimate one by one: the l.h.s of 
Lemma \ref{el}(d) is given by
\begin{align*}
   &\inte  (\nabla\varphi_i(x) - x) \cdot \nabla \psi_R(\nabla \varphi_i(x)) \eta_i(x)dx \\
   =&\inte  x \cdot \nabla \psi_R(x) \vri(x)dx -
   \inte  x \cdot \nabla \psi_R(\nabla \varphi_i(x)) \eta_i(x)dx \\
   =&\inte ( x \cdot v) \Psi_R(x)\vri(x) dx - 
  \inte  (x \cdot v )\Psi_R(\nabla \varphi_i(x))\eta_i(x) dx + O\left(\frac{1}{R}\right) \\
  = & \inte  (x \cdot v )\vri(x) dx - 
  \inte  (x \cdot v) \eta_i(x) dx + o(1) \ \ \ as \ R \rightarrow \infty.
  \end{align*}
  On the other hand, the r.h.s of Lemma \ref{el}(d) can be computed using the estimates
  \begin{align*}
  &|\Delta \psi_R(x)| = O\left(\frac{|x|}{R}\right) + O\Big(\frac{1}{R}\Big), \\
 &\frac{(\nabla \psi_R(x) - \nabla \psi_R(y))\cdot (x-y)}{|x-y|^2} 
 = \frac{1}{|x-y|^2} \Big[v \cdot (x-y) (\Psi_R(x) - \Psi_R(y)) \\
 &+ \big((x-y)\cdot v\big) \big((x-y)\cdot \nabla \Psi_R(x)\big) 
 + (y \cdot v) (\nabla \Psi_R(x) - \nabla \Psi_R(y)) \cdot (x-y)\Big]
 = \ O\left(\frac{1}{R}\right).
\end{align*}
Plugging $\psi_R$ in Lemma \ref{el1}(d) and letting $R \rightarrow \infty$ we obtain
\begin{align*}
 \Big(\si \inte x \vri(x) \ dx - \si \inte x\eta_i(x) \ dx \Big) \cdot v = 0
\end{align*}
for any $v \in \rt.$ 
\end{proof}

\begin{Lem} \label{wdestimate}
  Let $\et \in \ga$ and $\tau>0$ be given and assume that $\vr$ be a minimizer of  $\inf_{\bm{\rho} \in \ga} \G.$ If $\bm{\beta}$ is critical then
  \begin{align*}
   \dw^2 (\vr, \et) = M_2(\et) - M_2(\vr).
  \end{align*}
\end{Lem}

\begin{proof}
 Here we will approximate $|x|^2$ by smooth compactly supported functions. For that matter we choose $\psi_R(x) = |x|^2\Psi_R(x)$ in Lemma \ref{el1}(d). 
 A straight forward computation gives 
 \begin{align*}
  &\Delta \psi_R(x) = 4 \Psi_R(x) + O\Big(\frac{|x|}{R} + \frac{|x|^2}{R^2}\Big), \\
   &\frac{(\nabla \psi_R(x) - \nabla \psi_R(y))\cdot (x-y)}{|x-y|^2} 
 = \frac{1}{|x-y|^2} \Big[2|x-y|^2\Psi_R(x)-2y \cdot (x-y)(\Psi_R(x) - \Psi_R(y)) \\
 &+ |x|^2(x-y)\cdot (\nabla \Psi_R(x) - \nabla \Psi_R(y)) + 
 (|x|^2 - |y|^2) (x-y) \cdot \nabla \Psi_R(y)\Big] \\
 =& \ 2\Psi_R(x) + O\left(\frac{|y|}{R}\right)+O\left(\frac{|x|^2}{R^2}\right) + O\left(\frac{|x| + |y|}{R}\right).
 \end{align*}
 As a consequence, we see that 
 \begin{align*}
  &\si\inte  \Delta \psi_R(x) \vri(x)dx - \si\sj \frac{a_{ij}}{4\pi} \inte \inte  \frac{(\nabla \psi_R(x) - \nabla \psi_R(y))\cdot (x-y)}{|x-y|^2}\vri(x)\varrho_j(y)dx \\
  =& \ 4 \si\inte \Psi_R(x)\vri(x) - \si\sj \frac{a_{ij}}{2\pi}\inte \inte \Psi_R(x)\vri(x)\vrj(y) + O\Big(\frac{1}{R}\Big) \\
  =& \ \frac{\Lambda_I(\bm{\beta})}{2\pi} +o(1), \ \ as \ R \rightarrow \infty.
 \end{align*}
On the other hand 
  \begin{align*}
  &\inte (\nabla \varphi_i(x) - x) \cdot \nabla \psi_R(\nabla \varphi_i(x)) \eta_i(x)dx \\
 =& \ 2\inte|x|^2\Psi_R(x)\vri(x)dx + \inte  x \cdot \nabla\Psi_R(x) |x|^2\vri(x)dx  \\
 &-2\inte  x \cdot \nabla \varphi_i(x) \Psi_R(\nabla \varphi_i(x))\eta_i(x)dx 
 - \inte x \cdot \nabla\Psi_R(\nabla\varphi_i(x)) |\nabla \varphi_i(x)|^2\eta_i(x)dx.
\end{align*}
Note that as $\Psi_R$ vanishes outside $B(0,2R),$ the quantity $|y \nabla \Psi_R(y)|$ remains uniformly bounded for all $y \in \rt$
and converges point wise to $0$ as $R \rightarrow \infty.$ 
It follows that the second and forth terms of the above expression decay to zero as $R\rightarrow\infty$.  Indeed, $|x|^2\varrho_j$ is integrable by assumption so the second term vanish as $R\rightarrow\infty$ by the dominated convergence theorem. As for the forth term, we may estimate
\begin{align*}
 \int_{\R^2}x\cdot\nabla\Psi_R(\nabla\varphi_i))|\nabla\varphi_i|^2\eta_i(x)dx 
 &\leq \ C\int_{\R^2} |x||\nabla\varphi_i(x)|\eta_i(x) dx\\
 &\leq \frac{C}{2}\left(\int_{\R^2}
|x|^2\eta_i(x)dx + \int_{\R^2}|y|^2\varrho_i(y)dy\right) \ . 
 \end{align*}

By the assumed criticality of $\bm{\beta}$ and Lemma \ref{el}-(d)
  we conclude 
\begin{align} \label{wdr1}
\si\inte|x|^2\vri(x)dx = \si \inte  x \cdot \nabla \varphi_i(x) \eta_i(x)dx.
\end{align}
Using \eqref{wdr1} and the definition of the Wasserstein distance we get
\begin{align*}
 \dw^2(\vr, \et) &= \si \inte |x - \nabla\varphi_i(x)|^2 \eta_i(x) \ dx  \\
 &= M_2(\et) + M_2(\vr) - 2\si \inte  x \cdot \nabla \varphi_i(x) \eta_i(x)dx \\
 &= M_2(\et) - M_2(\vr) \ . 
\end{align*}
\end{proof}

\subsubsection{Higher moment estimates}
For the next Lemma, we introduce the de la Vall\'{e}e Poussin convex function. Let $\vu \in C^{\infty}([0,\infty); [0,\infty))$ be a convex function satisfying 
the growth conditions:
\begin{align}\label{p}
\begin{cases}
 &\vu(0)= \vu^{\prime}(0) = 0, \ r \mapsto\vu^{\prime}(r) \ is \ concave, \ \vu(r) \leq r\vu^{\prime}(r) \leq 2 \vu(r) \ for \ r >0, \\
& r \mapsto r^{-1}\vu(r) \ is \ concave, \ \lim_{r \rightarrow \infty}\frac{\vu(r)}{r} = \lim_{r\rightarrow \infty} \vu^{\prime}(r) = \infty.
\end{cases}
\end{align}
We like to give a special emphasize to the convex function $\vu$ and all its properties. It is one of the key step in obtaining uniform bound on the 
entropy of the iterates obtained through minimizing movement scheme (see section \ref{section mms} for the definition).
\begin{Lem}\label{better integrability}
Let $\et \in \ga$ and assume that $\si\vu(|x|^2)\eta_i \in L^1(\rt)$ where $\vu$ is a convex function satisfying all the properties in \eqref{p}. Then there exists a $\tau_0>0$ such that for any $\tau \in (0,\tau_0)$ if the 
minimization problem $\inf_{\bm{\rho} \in \ga} \G$ admits a solution and the minimizer $\vr$ satisfies $\si \vu(|x|^2)\vri \in L^1(\rt)$ then 
\begin{align*}
 \si \inte \vu(|x|^2)\vri(x) \ dx \leq (1 + C_0 \tau) \si \inte \vu(|x|^2)\eta_i(x) \ dx+ C_0 \tau,
\end{align*}
where $C_0$ is a constant depending only on  $A, \bm{\beta}$ and $\tau_0.$ 
\end{Lem}
The proof of Lemma \ref{better integrability} is  rather technically cumbersome. For the moment the validity of the lemma
is taken for granted. We will give a proof in the appendix. We emphasize on the fact that the constant $C_0$ does not depend on the 
$L^1$ norm of $\si\vu(|x|^2)(\vri + \eta_i).$

 \section{The critical case: Existence of minimizers and blow-up behaviour}\label{section critical}
 At this present section, we are going to establish a sufficient criterion for the existence of minimizers of $\G$ at the critical mass regime.
 We found that there are only two possibilities for a minimizing sequence. Either they converge to a minimizer, or, if they were to blow-up, they must concentrate in the form of a Dirac delta measure at a common point.

Our starting point is the following functional inequality, which turns out to be very crucial in the analysis to come.
 \begin{Lem} \label{bound1}
 Assume $\bm{\beta}$ is critical, then
 \begin{align} \label{fi}
  \inf_{\bm{\rho} \in \ga} \G \leq \inf_{\bm{\rho} \in \ga} \F + \frac{1}{2\tau}M_2(\bm{\eta}).
 \end{align}
\end{Lem}

\begin{proof}
 Let $\{\bm{\rho}^m\} \subset \ga$ be a minimizing sequence for $\inf_{\bm{\rho} \in \ga} \F.$ Choose a sequence $R_m
 \rightarrow \infty$ and let 
 $\bm{\tilde \rho}^m (x) := (R_m)^2 \bm{\rho}^m(R_m x)$.
 It follows that  $ M_2(\bm{\tilde \rho}^m) \rightarrow 0.$ Since $\bm{\beta}$ is critical, 
 $\mathcal{F}(\bm{\tilde \rho}^m)=\mathcal{F}(\bm{\rho}^m)$, 
so is also a minimizing sequence.
 
 The proof follows from the straightforward inequality
 \begin{align*}
  \inf_{\bm{\rho} \in \ga} \G \leq \mathcal{G}_{\bm{\eta}}(\bm{\tilde \rho}^m) 
 &= \mathcal{F}(\bm{\tilde \rho}^m) + \frac{1}{2\tau}\dw^2(\bm{\tilde \rho}^m, \bm{\eta}) \\
  &=  \inf_{\bm{\rho} \in \ga}\mathcal{F}(\bm{\rho}) + \frac{1}{2\tau}M_2(\bm{\eta}) + o(1), \ as \ m \rightarrow \infty.
   \end{align*}
 \end{proof}
 
 Next we state the main result of this section.

 \begin{Th} \label{mms existance thm}
   Assume $A$ and $\bm{\beta}$ satisfies Assumption \ref{AssLambda}  and $\bm{\eta} \in \ga.$ Then either one of the following alternative holds: 
  \begin{itemize}
   \item[(a)] equality holds in \eqref{fi},
   \item[(b)] for every minimizing sequence $\{\bm{\rho}^m\}$ to $\inf_{\bm{\rho} \in \ga} \G$, the entropy $\mathcal{H}(\bm{\rho}^m)$ is uniformly bounded.
  \end{itemize}
  In particular, if $\mathcal{F}(\bm{\eta}) < \inf_{\bm{\rho} \in \ga} \F + \frac{1}{2\tau}M_2(\bm{\eta})$ then the minimization problem $\inf_{\bm{\rho} \in \ga} \G$ 
  admits a solution $\vr \in \ga.$ Moreover, the minimizer $\vr \in \ga$ satisfies
  \begin{align*}
   \mathcal{F}(\vr) < \inf_{\bm{\rho} \in \ga} \F + \frac{1}{2\tau}M_2(\vr).
  \end{align*}

\end{Th}

\begin{proof}
 The proof of the theorem is divided into several steps.
 
 \vspace{0.2 cm}
 
 \noindent
 {\bf Step 1:} Let $\{\bm{\rho}^m\}$ be a minimizing sequence for $\inf_{\bm{\rho} \in \ga} \G.$
By Proposition \ref{propertyf}(b) and \eqref{1}, the second moment of $\bm{\rho}^m: M_2(\bm{\rho}^m)$ remains uniformly bounded and hence the 
measures $\bm{\rho}^m$ are tight. By Prokhorov's theorem, $\bm{\rho}^m \overset{\ast}{\rightharpoonup} \bm{\rho}^*$ in the weak $*$ topology of measures. In the following we will prove that either all the components
of $\bm{\rho}^m$ concentrates in the form of a Dirac delta measure at a common point $v_0$ or, the entropy $\mathcal{H}(\bm{\rho}^m)$ remains uniformly bounded. In other words,
either $\rho_i^* = \beta_i \delta_{v_0}$ for every $i \in I,$ or, $\mathcal{H}(\bm{\rho}^m) \leq C$ for some constant independent of $m.$ 
The concentration property of $\bm{\rho}^m$ together with the lower semi-continuity of $\dw$ will give the conclusion (a). 

\vspace{0.2 cm}

\noindent
{\bf Step 2:} If at least one component does not concentrate then no component concentrate at all.

\vspace{0.2 cm}

With out loss of generality we can assume $\rho_n^{*}$ is not a Dirac mass and $0 \in supp (\rho_n^*)$. If $0 \not\in supp(\rho_n^*)$ we can work with 
any point within the support of $\rho_n^*$ and the same argument can be carried over. We define 
\begin{align*}
 \alpha_n^{*}(r) := \int_{B_r} d\rho_n^*. 
\end{align*}
Then $\lim_{r \rightarrow \infty} \alpha_n^*(r) = \beta_n$ and $\rho_n^*$ is not a Dirac mass is equivalent to saying $\lim_{r \rightarrow 0+}\alpha_n^*(r) < \beta_n.$ 
Since $\alpha_n^{*}$ is monotone, the points of discontinuities of $\alpha_n^*$ is at most countable. We can choose a point $r_1$ small enough such that 
$0 <\alpha_n^*(r_1) < \beta_n$ and $r_1$ is a point of continuity of $\alpha_n^*.$

We choose $\delta > 0$ small so that $0 < \alpha_n^*(r_1) - \delta < \alpha_n^*(r_1) + \delta < \beta_n.$ Later we will make further smallness assumption on $\delta.$
By continuity of $\alpha_n^*$ at $r_1$ and monotonicity, we can choose $r_0 < r_1 < r_2$ such that 
\begin{align*}
 0 < \alpha_n^*(r_1) - \delta < \alpha_n^*(r_0)\leq \alpha_n^*(r_1) \leq \alpha_n^*(r_2) < \alpha_n^*(r_1) + \delta < \beta_n.
\end{align*}
Since $\rho_n^m \overset{\ast}{\rightharpoonup} \rho_n^*,$ for large enough $m$ we have

\begin{align*}
 \alpha_n^*(r_1) - \delta \leq \int_{B_{r_0}} &\rho_n^m(x) \ dx, \ \beta_n - \alpha_n^*(r_1) - \delta \leq \int_{B_{r_2}^c} \rho_n^m(x) \ dx,\\
 &\int_{B_{r_2} \backslash B_{r_0}} \rho_n^m(x) \ dx \leq 2\delta.
\end{align*}
Applying logarithmic HLS inequality in three regions $B_{r_0 + \epsilon}, B_{r_2 - \epsilon}^c$ and $B_{r_2} \backslash B_{r_0}$ and proceeding as in \cite[Lemma $3.1$]{BCM} we conclude that 
\begin{align}\label{mul}
 \kappa_n^m\inte \rho_n^{m}(x) \ln \rho_n^{m}(x) + 2 \inte \inte \rho_n^m(x) \ln|x-y| \rho_n^m(y) \ dx dy \geq C,
\end{align}
where $\kappa_n^m = max \{a^m_1 + a_2^m, a_2^m + a_3^m\}$, $\epsilon = \frac{1}{3}(r_2-r_0$) and
\begin{align*}
 a_1^m = \int_{B_{r_0 + \epsilon}} \rho_n^{m}(x) \ dx, \ a_2^m = \int_{B_{r_2} \backslash B_{r_0}} \rho_n^{m}(x) \ dx, \ 
 a_3^m = \int_{B_{r_2 - \epsilon}^c} \rho_n^{m}(x) \ dx.
\end{align*}
By choosing $\delta < \frac{1}{6} \min\{\alpha_n^*(r_1), \beta_n - \alpha_n^*(r_1)\}$ we see that $\kappa_n^m < \beta_n - \delta$ for sufficiently large $m.$

\vspace{0.2 cm}

\noindent
{\bf Step 3:} Uniform boundedness on the entropy.

\vspace{0.2 cm}

Since $\bm{\rho}^m$ is a minimizing sequence, for sufficiently large $m$
\begin{align}\label{rewrite}
 \si \inte \rho_i^m\ln \rho_i^m \ dx + \si \sj \frac{a_{ij}}{4\pi} \inte \inte \rho_i^m(x) \ln |x-y| \rho_j^m(y) \ dxdy
 \leq \mathcal{F}(\et) + 1.
\end{align}

Multiplying \eqref{mul} by $\frac{a_{nn}}{8\pi}$ and subtracting from \eqref{rewrite} we get 
\begin{align}\label{deduce100}
 \si b_i \inte \rho_i^m\ln \rho_i^m \ dx + \si \sj \frac{\tilde a_{ij}}{4\pi} \inte \inte \rho_i^m(x) \ln |x-y| \rho_j^m(y) \ dxdy \leq C,
\end{align}
where 

\begin{align*}
 b_i = 
 & \begin{cases}
  1, \ \ \ \ \ \ \ \ \ \ \ \ \ \ \ \ if \ i \neq n,\\
  (1- \frac{a_{nn}\kappa_n^m}{8\pi}), \ \ if \ i = n,
 \end{cases}
& & 
\tilde a_{ij} =
  \begin{cases}
  0, \ \ \ \ if \ i = j = n, \\
  a_{ij}, \ \ otherwise.
 \end{cases}
\end{align*}
 Now we claim that the mass $\bm{\beta}$ is sub-critical with respect to $\bm{b}$ and the interaction matrix $(\tilde a_{ij}),$ i.e., 
 $\Lambda_J(\bm{\beta}; \bm{b}) \geq \bar\varepsilon > 0$ for all $J \subset I$ and for some $\bar \varepsilon$ (see \eqref{lambda b beta}). Let $J \subset I$ be any subset. If $n  \notin J$ then 
 there is nothing to prove, because, in that case $\Lambda_J(\bm{\beta}; \bm{b}) = \Lambda_J(\bm{\beta})>0,$
 according to our assumption. If $n \in J$ then 
 \begin{align} \label{provingsub1}
  \Lambda_J(\bm{\beta};\bm{b}) &= 8\pi \sum_{i \in J} b_i\beta_i - \sum_{i \in J} \sum_{j \in J}\tilde a_{ij} \beta_i \beta_j \notag \\
  &= \Lambda_{J\backslash \{n\}}(\bm{\beta}) - 2 \sum_{i \in J\backslash \{n\}} a_{in} \beta_i \beta_n + 8\pi b_n \beta_n.
 \end{align}
By our assumption $\Lambda_J(\bm{\beta}) > 0$ if $J \neq I$ and equal to zero if $J = I.$ In either case the condition $\Lambda_J(\bm{\beta}) \geq 0$
can be rewritten as 
\begin{align}\label{provingsub2}
 \Lambda_{J\backslash \{n\}}(\bm{\beta}) - 2 \sum_{i \in J\backslash \{n\}} a_{in} \beta_i \beta_n \geq a_{nn}\beta_n^2 - 8\pi \beta_n.
\end{align}
Using \eqref{provingsub2} in \eqref{provingsub1} we get 
\begin{align*}
 \Lambda_J(\bm{\beta}, \bm{b}) &\geq a_{nn}\beta_n^2 - 8\pi \beta_n + 8\pi\left(1- \frac{a_{nn}\kappa_n^m}{8\pi}\right)\beta_n \\
 &= a_{nn}\beta_n(\beta_n - \kappa_n^m)\geq a_{nn}\delta \beta_n > 0
\end{align*}
provided $a_{nn} > 0$ which is valid owing to our assumption. According to Proposition \ref{propertyf}(d) and Remark \ref{remark lhls} and the uniform 
second moment bound we deduce from \eqref{deduce100} that the entropy $\mathcal{H}(\bm{\rho}^m)$ is uniformly bounded.

At this stage we know that if at least one component does not concentrate then we have uniform bound on the entropy. As a consequence 
of Proposition \ref{gnu proposition}(c), the existence of a minimizer follows. 

Next we will deal with the case when concentration does happen. Assume $\rho_i^* = \beta_i \delta_{v_i}$  for some points $v_i \in \rt.$

\vspace{0.2 cm} 

\noindent
{\bf Step 4:} $v_1 = \cdots = v_n = 0.$

\vspace{0.2cm}

Let $I_1 \subset I$ be the set of all indices such that $v_i = v_1$ for all $i \in I_1.$ It is enough to show that $I_1 = I.$ We set $I_2 = I \backslash I_1.$ 
Let us take $i \in I_1$ and $j \in I_2.$ We claim that for any $\iota >0$ small, we can estimate  
\begin{align} \label{logbbd}
 \inte \inte \rho_i^m(x) \ln |x-y| \rho_j^m(y) \ dxdy \geq -\alpha_j^m \inte \rho_i^m \ln \rho_j^m - \alpha_i^m \inte 
 \rho_j^m \ln \rho_j^m - C,
\end{align}
for $m$ large enough and where $\max \{\alpha_i^m, \alpha_j^m \}< \iota$. To do so, we set $\epsilon = \min\{\frac{1}{10}|v_i - v_j|, 1\}$ and
decompose the integral $\inte \inte \rho_i^m(x) \ln |x-y| \rho_j^m(y) \ dxdy$
into three pieces $A, B$ and C: where
\begin{align*}
 \displaystyle{ A := \int \int_{\{|x-y| < \epsilon\}} \cdot \ , \ B := \int \int_{\{\epsilon \leq |x-y| < R\}} \cdot \ , \ C := \int \int_{\{|x-y|>R\}}\cdot}
\end{align*}
and $R > 1$ is large. The integral $B$ is  easy to estimate by the $L^1$ bound of $\rho_i^m, \rho_j^m$.  The integral $C$ is estimated by the bound on $M_2(\bm{\rho^m})$.

To estimate the integral $A$ we see that the region $\{|x - y| < \epsilon\}$ is contained in the union of $\{|x-y|< \epsilon\} \cap \{x \in B_{\epsilon}(v_i)^c\}$
and $\{|x - y| < \epsilon\} \cap \{y \in B_{\epsilon}(v_j)^c\}.$ Indeed, if $x \in B_{\epsilon} (v_i)$ and $y \in B_{\epsilon}(v_j)$ then 
$|x - y| \geq |v_i - v_j| - |x - v_i| - |x - v_j| \geq 8\epsilon.$ We only estimate the region $\{|x - y| < \epsilon\} \cap \{y \in B_{\epsilon}(v_j)^c\},$ 
the other being verbatim copy with the role of $i$ and $j$ interchanged.
Choose $\alpha < 1$, then using the inequality $ab \leq b\ln b + e^a,$ we deduce 
\begin{align*}
 &\left|\int \int_{\{|x - y| < \epsilon\} \cap \{y \in B_{\epsilon}(v_j)^c\}} \rho_i^m(x) \ln |x-y| \rho_j^m(y) \ dxdy \right| \\
 = & \ \int \int_{\{|x - y| < \epsilon\} \cap \{y \in B_{\epsilon}(v_j)^c\}} (\alpha^{-1}\rho_i^m(x)) \left(\alpha\ln \frac{1}{|x-y|}\right) \rho_j^m(y) \ dxdy\\
 \leq & \int \int_{\{|x - y| < \epsilon\} \cap \{y \in B_{\epsilon}(v_j)^c\}} \rho_j^m(y) \left[\alpha^{-1}\rho_i^m(x)\ln \rho_i^m(x) 
 + \frac{1}{|x -y|^{\alpha}} \right]\ dxdy \\
 \leq & \ \alpha^{-1}\left(\int_{B_{\epsilon}(v_j)^c} \rho_j^m(x) \ dx\right) \inte \rho_i^m(x) |\ln \rho_i^m(x)| \ dx + C \\
 \leq & \ \alpha^{-1}\left(\int_{B_{\epsilon}(v_j)^c} \rho_j^m(x) \ dx\right) \inte \rho_i^m(x) \ln \rho_i^m(x) + C,
\end{align*}
where in the last line we have used the inequality  \cite[Lemma $2.6$]{BD}:
\begin{align*}
 \inte \rho(x)|\ln \rho(x)| \ dx \leq \inte \rho(x)\ln \rho(x) + M_2(\rho) + 2\ln (2\pi)\inte \rho(x) \ dx + \frac{2}{e}.
\end{align*}
Taking into account the integrals $A,B$ and $C$ we conclude that \eqref{logbbd} holds with 
$\alpha_l^m=\alpha^{-1}\int_{B_{\epsilon}(v_l)^c} \rho_l^m(x)\ dx. $ Since $\rho_l^m \overset{\ast}{\rightharpoonup} \beta_l \delta_{v_l}$, for lagre enough $m$
we can make $\alpha_l^m < \iota.$

We can rewrite the upper bound condition $\f(\bm{\rho}^m) \leq C$ in the following way:
\begin{align}\label{fsplit}
 &\sum_{i \in I_1}\inte \rho_i^m \ln \rho_i^m \ dx + \sum_{i \in I_1}\sum_{j \in I_1}\frac{a_{ij}}{4\pi}\inte \inte \rho_i^m(x) \ln |x-y| \rho_j^m(y) \ dxdy \notag\\
  &+\sum_{i \in I_2} \inte\rho_i^m \ln \rho_i^m \ dx + \sum_{i \in I_2}\sum_{j \in I_2}\frac{a_{ij}}{4\pi}\inte \inte \rho_i^m(x) \ln |x-y| \rho_j^m(y) \ dxdy \notag \\
  &+ \sum_{i \in I_1}\sum_{j \in I_2}\frac{a_{ij}}{2\pi}\inte \inte \rho_i^m(x) \ln |x-y| \rho_j^m(y) \ dxdy \leq C.
\end{align}

Using \eqref{logbbd} in \eqref{fsplit} we conclude that
\begin{align*}
&\left[\sum_{i \in I_1} (1- b_i^m) \inte\rho_i^m \ln \rho_i^m \ dx + \sum_{i \in I_1}\sum_{j \in I_1}\frac{a_{ij}}{4\pi}\inte \inte \rho_i^m(x) \ln |x-y| \rho_j^m(y) \ dxdy \right]\notag\\
  &+ \left[\sum_{i \in I_2} (1- b_i^m)\inte\rho_i^m \ln \rho_i^m \ dx + \sum_{i \in I_2}\sum_{j \in I_2}\frac{a_{ij}}{4\pi}\inte \inte \rho_i^m(x) \ln |x-y| \rho_j^m(y) \ dxdy \right]\leq C. 
\end{align*}
where $b_i^m = \sum_{j \in I_2} \frac{a_{ij} \alpha_j^m}{2\pi}$ if $i \in I_1$ and when $i \in I_2, \ b_i^m = \sum_{j \in I_1} \frac{a_{ij} \alpha_j^m}{2\pi}.$  
At least one of the quantities within the brackets must be bounded above. We assume
\begin{align*}
\sum_{i \in I_1}(1- b_i^m) \inte\rho_i^m \ln \rho_i^m \ dx + \sum_{i \in I_1}\sum_{j \in I_1}\frac{a_{ij}}{4\pi}\inte \inte \rho_i^m(x) \ln |x-y| \rho_j^m(y) \ dxdy \leq C.
\end{align*}
Since $\Lambda_{J}(\bm{\beta}) > 0,$ for all $J \subset I_1$ and $b_i^m$ can be made as small as one desires, by Remark \ref{remark lhls}
we obtain uniform bound on the entropy of $\rho_i^m$ for all $i \in I_1.$ This contradicts that $\bm{\rho}*$ is a Dirac measure. Hence $I = I_1$ and we denote the 
common blow-up point by $v_0.$ By second moment bound we conclude
\begin{align*}
 (\si \beta_i)v_0 = \lim_{m \rightarrow \infty}\si \inte x \rho_i^m(x) \ dx  = 0.
\end{align*}

\vspace{0.2 cm}

\noindent
{\bf Step 5:} Concentration implies equality holds in the functional inequality. 

\vspace{0.2 cm}

Since $\rho_i^m \overset{\ast}{\rightharpoonup} \beta_i \delta_0$ for all $i \in I,$ by lower semi-continuity Lemma \ref{semicontinuity of dw} we conclude that 
\begin{align*}
 \inf_{\bm{\rho} \in \ga} \G &= \lim_{m \rightarrow \infty} \left (\f(\bm{\rho}^m) + \frac{1}{2\tau} \dw^2 (\bm{\rho}^m, \et)\right)\\
 &\geq \inf_{\bm{\rho} \in \ga} \F + \frac{1}{2\tau} \liminf_{m \rightarrow \infty} \dw^2 (\bm{\rho}^m, \et)\\
 &\geq \inf_{\bm{\rho} \in \ga} \F + \frac{1}{2\tau} \dw^2(\bm{\beta}\delta_{\bm{0}}, \et) \\
 & = \inf_{\bm{\rho} \in \ga} \F + \frac{1}{2\tau}M_2(\et).
\end{align*}
Using Lemma \ref{bound1} we obtain the equality in the functional inequality.

\vspace{0.2 cm}

\noindent
{\bf Step 6:} If equality in \eqref{fi} holds, then there exists a blowing up minimizing sequence.

\vspace{0.2 cm}

Let $\{\bm{\rho}^m\} \subset \ga$ be a minimizing sequence for $\f$ and $\bm{\tilde \rho}^m$ be defined as in the proof of Lemma \ref{bound1}. It is easy to see that $\mathcal{H}(\bm{\tilde \rho}^m) \rightarrow \infty$ as $m \rightarrow \infty.$ Then 
\begin{align*}
  \inf_{\bm{\rho} \in \ga} \F + \frac{1}{2\tau}M_2(\bm{\eta})= \inf_{\bm{\rho} \in \ga} \G \leq \mathcal{G}_{\et}(\bm{\tilde \rho}^m)
  = \inf_{\bm{\rho} \in \ga} \F + \frac{1}{2\tau}M_2(\bm{\eta}) + o(1),
\end{align*}
as $m \rightarrow \infty.$ Hence,  $\bm{\tilde \rho}^m$ is a concentrating minimizing sequence.

\vspace{0.2 cm}

\noindent
{\bf Step 7:} Concluding the proof of the theorem.

\vspace{0.2 cm}

In view of the first part of the theorem, if $\et$ satisfies $\f(\et) < \inf_{\bm{\rho} \in \ga} \F + \frac{1}{2\tau} M_2(\et)$ then there exists a minimizer of 
$\inf_{\bm{\rho} \in \ga} \G,$ and let us denote the minimizer by $\vr.$ Utilizing Lemma \ref{wdestimate}, the obtained minimizer satisfies
\begin{align*}
 \f(\vr) = \mathcal{G}_{\bm{\eta}}(\vr) - \frac{1}{2\tau}\dw^2(\vr, \et) &\leq \f(\et) - \frac{1}{2\tau}\dw^2(\vr, \et)\\
 & < \inf_{\bm{\rho} \in \ga} \F + \frac{1}{2\tau} M_2(\et) - \frac{1}{2\tau} \Big(M_2(\et) - M_2(\vr)\Big) \\
 &= \inf_{\bm{\rho} \in \ga} \F + \frac{1}{2\tau}M_2(\vr).
\end{align*}
This completes the proof of the theorem.
\end{proof}

\begin{Rem} \label{mms existence remark}
 A careful tracing back of the proof of Theorem \ref{mms existance thm}, in particular, step $2$ to step $4$ confirms that:
 suppose we have a sequence $\{\bm{\rho}^m\} \subset \ga$ such that the energy $\f(\bm{\rho}^m)$ and the second moment $M_2(\bm{\rho}^m)$ are uniformly bounded above then one of the 
 following alternative holds
 \begin{itemize}
  \item either the entropy is uniformly bounded, or
  \item all the components of $\bm{\rho}^m$ concentrates at the origin in the form of a Dirac delta measure.
 \end{itemize}
This observation will be useful in the next section while trying to obtain uniform entropy estimates.
\end{Rem}

\section{A priori Estimates} \label{section apriori}
In this section we will obtain uniform estimates on the interpolates $\bm{\rho}_{\tau}^k, k \in \mathbb{N}$ obtained by the MM-scheme \eqref{mms}.
Recall that if $\tau \in (0 ,\tau^*)$ then MM-scheme is well defined.

\subsection{Uniform bound on the second moment and the Wasserstein distance}
\begin{Lem} \label{apriori}
 For any $\tau \in (0, \tau^*)$ and any positive integer $k$ there holds 
 \begin{align} \label{ap1}
  M_2(\bm{\rho}_{\tau}^k) + \frac{1}{2\tau}\sum_{l=1}^{k} \dw^2(\bm{\rho}_{\tau}^{l}, \bm{\rho}_{\tau}^{l-1})
 \leq \f(\bm{\rho}^0) - C_{\mbox{{\tiny{LHLS}}}}(\bm{\beta}) + M_2(\bm{\rho}^0).
 \end{align}
\end{Lem}
\begin{proof}
 For every $l \in \{1,\ldots,k\},$ the minimizing property of $\rl$ gives
 \begin{align*}
  \f(\rl) + \frac{1}{2\tau}\dw^2(\rl, \bm{\rho}^{l-1}_{\tau}) \leq \f(\bm{\rho}^{l-1}_{\tau}). 
 \end{align*}
Summing over $l \in \{1,\ldots,k\}$ we obtain
\begin{align} \label{ap2}
 \f(\rk) + \frac{1}{2\tau}\sum_{l=1}^{k} \dw^2(\bm{\rho}_{\tau}^{l}, \bm{\rho}_{\tau}^{l-1}) \leq \f(\bm{\rho}^0).
\end{align}
Since $\bm{\beta}$ is critical, $\f(\rk) \geq C_{\mbox{{\tiny{LHLS}}}}(\bm{\beta}),$ and hence 
\begin{align} \label{ap4}
 \frac{1}{2\tau}\sum_{l=1}^{k} \dw^2(\bm{\rho}_{\tau}^{l}, \bm{\rho}_{\tau}^{l-1}) \leq \f(\bm{\rho}^0) - C_{\mbox{{\tiny{LHLS}}}}(\bm{\beta}).
\end{align}
On the other hand, by Lemma \ref{wdestimate}
\begin{align*}
 M_2(\rk) \leq M_2(\rka) \leq \cdots \leq M_2(\bm{\rho}_{\tau}^0) = M_2(\bm{\rho}^0),
\end{align*}
completing the proof.
\end{proof}


\subsection{Uniform bound on the entropy}

To obtain uniform bound on the entropy, we are going to use a refined version of de la Vell\'{e}e Poussin's lemma on the gain on integrability. Let us first recall the lemma
\begin{Lem}[de la Vell\'{e}e Poussin] \label{de la lemma}
 Let $\mu$ be a non-negative measure on $\rt$ and $\mathcal{Z} \subset L^1(\rt;\mu).$ The set $\mathcal{Z}$ is uniformly integrable in $L^1(\rt;\mu)$ if and only
if $\mathcal{Z}$ is uniformly bounded in $L^1(\rt;\mu)$ and there exists a convex function $\vu \in C^{\infty}([0,\infty); [0,\infty))$ satisfying all the properties listed 
in \eqref{p} and 
\begin{align*}
 \sup_{g \in \mathcal{Z}} \inte \vu(g(x)) \ d\mu(x) < +\infty. 
\end{align*}
\end{Lem}
Such gain of integrability result was first
implemented by de la Vall\'{e}e Poussin \cite{Poussin} in the context of $L^1$-weak compactness of a 
family of functions, but with out the regularity and concavity of $\vu^{\prime}.$
We refer to \cite{Poussin2, dela, LMdela, Ldela} concerning the regularity and the other mentioned properties of $\vu$.

We are going to apply de la Vell\'{e}e Poussin lemma for a suitably constructed measure $\mu$ and to the family $\mathcal{Z}$ containing only one element $g(x) = |x|^2.$ 
We remark that it is, in general, impossible to  iterate  the higher integrability information on the initial data to the sequence  $\rk$ obtained by the MM-scheme. In particular,  if 
we know $\si \vu(|\cdot|^2)\rho_{\tau,i}^0(\cdot) \in L^1(\rt),$ then it is not evident  that $\si \vu(|\cdot|^2)\rho_{\tau,i}^1(\cdot)$ will also be uniformly integrable.

\begin{Lem} \label{apriori entropy}
  For every $T \in (0,\infty)$ there exists a constant $\mathcal{C}_{ap}(T)$ 
 such that for each $\tau \in (0,\tau^*)$ and positive integers $k$ satisfying $k\tau \leq T$ there holds 
 \begin{align*} 
  \si \inte \rho_{\tau,i}^k|\ln \rho_{\tau,i}^k| \ dx\leq \mathcal{C}_{ap}(T).
 \end{align*}
\end{Lem}
\begin{proof}
 Fix $T,$ and assume that the claim is false. Then there exists a sequence $\tau_m \rightarrow 0+$ as $m \rightarrow \infty,$ and $k_m \in \mathbb{N}$ such that $k_m \tau_m \leq T$ 
 and $\mathcal{H}(\rkm) \rightarrow \infty.$ By \eqref{ap2} we know that 
 \begin{align*}
  \f(\rkm) \leq \f(\bm{\rho}^0), \ \ for \ all \ m.
 \end{align*}
Moreover, recall that by Lemma \ref{zero com} each $\rkm$ satisfies zero center mass condition i.e., $\si \inte x \rho_{\tau_m,i}^{k_m}(x) dx = 0.$
Proceeding as in the proof of Theorem \ref{mms existance thm} (and heeding Remark \ref{mms existence remark}) we conclude that $\rho_{\tau_m,i}^{k_m} \overset{\ast}{\rightharpoonup} \beta_i \delta_0$ in the 
weak*-topology of measures. Next we are going to obtain uniform integrability of the second moment of $\rkm.$ We set 
\begin{align*}
 \mu = \sum_{m=1}^{\infty} \frac{1}{(k_m+1)2^{m+1}} \sum_{l=0}^{k_m}\si \rho_{\tau_m,i}^{l} .
\end{align*}
Then taking into account Lemma \ref{apriori}, we obtain 
\begin{align*}
 M_2(\mu) = \sum_{m=1}^{\infty} \frac{1}{(k_m+1)2^{m+1}} \sum_{l=0}^{k_m}M_2(\bm{\rho}_{\tau_m}^{l} )\leq 2M_2(\bm{\rho}^0).
\end{align*}
For the above choice of the measure $\mu,$ we apply de la Vell\'{e}e Poussin lemma to $g(x) = |x|^2.$ There exists a smooth convex function $\vu$
satisfying all the properties enumerated in \eqref{p} such that $\inte \vu(|x|^2) \ d\mu(x) =\mathcal{C}_0< +\infty.$ As a consequence we obtain 
\begin{align*}
 \si \inte \vu(|x|^2) \rho_{\tau_m,i}^{l} \ dx \leq (k_m + 1)2^{m+1}\mathcal{C}_{0}, \ \ for \ all \ l \in \{0, \ldots, k_m\}.
\end{align*}
Now we use Lemma \ref{better integrability} to obtain uniform estimates on $\si \inte \vu(|x|^2) \rho_{\tau_m,i}^{k_m}(x) \ dx.$ Applying Lemma 
\ref{better integrability} with $\vr = \rkm$ and $\et = \rkam$ we see that 
\begin{align*}
 \si \inte \vu(|x|^2) \rho_{\tau_m,i}^{k_m}(x) \ dx \leq (1 + C_0 \tau_m)\si \inte \vu(|x|^2) \rho_{\tau_m,i}^{k_m-1}(x) \ dx + C_0\tau_m.
\end{align*}
Iterating this process $k_m$ times we obtain
\begin{align} \label{iterbbd}
 \si \inte \vu(|x|^2) \rho_{\tau_m,i}^{k_m}(x) \ dx \leq& (1 + C_0 \tau_m)^{k_m}\si \inte \vu(|x|^2) \rho_{\tau_m,i}^{0}(x) \ dx \notag \\
 & \ + C_0\tau_mk_m(1 + C_0 \tau_m)^{k_m}.
\end{align}
Since $k_m\tau_m \leq T,$ and $\rho_{\tau_m,i}^{0} = \rho_i^0$ we deduce from \eqref{iterbbd}
\begin{align} \label{higher moment}
 \si \inte \vu(|x|^2) \rho_{\tau_m,i}^{k_m}(x) \ dx \leq e^{C_0T}\si \inte \vu(|x|^2) \rho_{i}^{0}(x) \ dx + C_0Te^{C_0T}.
\end{align}

Taking into account \eqref{higher moment} and the the weak* convergence to the Dirac measure we conclude $M_2(\rkm) \rightarrow 0$ as $m \rightarrow \infty$. On the other hand, by Lemma \ref{wdestimate}
\begin{align} \label{q1}
 M_2(\rkam) - M_2(\rkm) &= \dw^2(\rkm, \rkam) \notag \\
 & \ \ \vdots \notag \\
 M_2(\bm{\rho}^0_{\tau_m}) - M_2(\bm{\rho}^1_{\tau_m}) &= \dw^2(\bm{\rho}^1_{\tau_m}, \bm{\rho}^0_{\tau_m}).
\end{align} 
Adding all the terms in \eqref{q1} and using $\bm{\rho}^0_{\tau_m} =\bm{\rho}^0$  and Lemma \ref{apriori} we get
\begin{align*}
 M_2(\bm{\rho}^0) = M_2(\rkm) + \sum_{l=1}^{k_m} \dw^2 (\bm{\rho}_{\tau_m}^l, \bm{\rho}_{\tau_m}^{l-1}) = M_2(\rkm) + O(\tau_m) \rightarrow 0
\end{align*}
as $m \rightarrow \infty,$ contradicting $\bm{\rho}^0 \in \ga.$ This contradiction appeared because we assumed the entropy is not uniformly bounded. Hence,
our assumption was wrong and the lemma is proved.
\end{proof}

\subsection{A priori estimates on the time interpolation}
\subsubsection{Time interpolation} We define the piece wise constant time dependent interpolation
\begin{align*}
 \bm{\rho}_{\tau}(t) = \rk, \ \ if \ t \in ((k-1)\tau, k\tau], \ \ k \geq 1.	
\end{align*}

Recall that the Newtonian potential of $\rki$ is defined by 
\begin{align}\label{np}
 u_{\tau,i}^k(x) = -\frac{1}{2\pi}\inte \ln|x-y| \rki(y)dy, \ \ for \ i=1,\ldots,n.
\end{align}
By definition of $\bm{\rho}_{\tau}(t),$ we have $u_{\tau, i}(t) = u_{\tau, i}^k(t)$ for all $t \in ((k-1)\tau, k\tau], \ \ k \geq 1.$

As a consequence of the a priori estimates obtained in Lemma \ref{apriori}, Lemma \ref{apriori entropy} and the regularity estimates stated 
in Lemma \ref{regularity} we obtain the following:
\begin{Lem}\label{eti}
 For every $T>0,$ there exists a finite constant $\mathcal{C}(T) = \mathcal{C}(T, \bm{\rho}^0)>0,$ depending only on the time $T$ and the initial data $\bm{\rho}^0,$ 
 such that for every $\tau \in (0,\tau^*)$ we have

 \begin{itemize}
 \item[(a)] uniform entropy and second moment bound
\begin{align*}
 \sup_{t \in [0,T], \tau \in (0,\tau^*)} \si\left(\inte \rho_{\tau,i}(t)|\ln \rho_{\tau,i}(t)| + M_2(\rho_{\tau,i}(t))\right) \leq \mathcal{C}(T).
 \end{align*}
\item[(b)] Fisher information bound
\begin{align*}
\si \int_0^{T}\inte (\rti(t))^2 \ dxdt 
+ \si\int_{0}^{T} \inte \left|\frac{\nabla \rti(t)}{\rti(t)}\right|^2\rti(t) \ dxdt \leq \mathcal{C}(T).
\end{align*}
\item[(c)] Continuity in time bound: there exists $\mathcal{C}$
 such that
\begin{align*}
\dws(\rho_{\tau,i}(t), \rho_{\tau,i}(s)) \leq \mathcal{C}\left(\sqrt{|t-s|} + \sqrt{\tau}\right), \ s,t \in [0,\infty) \ and \ i=1,\ldots,n.
 \end{align*}
\end{itemize}

 The Newtonian potentials $u_{\tau,i}(t), t\in [0,T]$ satisfy the following uniform 
 $H^1_{loc}$ and $L^2-H^2_{loc}$ estimate:
  \begin{itemize}
 \item[(d)] For each $R>0,$ there exists a constant $\mathcal{C}(R,T)$ depending only on $R$ and $\mathcal{C}(T)$ such that 
 \begin{align*}
  \si ||u_{\tau,i}(t)||_{H^1(B_R)} +  \si\int_0^T ||u_{\tau, i}(t)||_{H^2(B_R)}^2 \ dt \leq \mathcal{C}(R,T).
 \end{align*}
\end{itemize}
 \end{Lem}
 
 We again emphasize on the fact that once we have the uniform entropy and second moment bound, the remaining estimates stated in the above theorem 
 can be obtained proceeding as in the case of sub-critical $\bm{\beta}$. The proof of $(a)$ follows from the a priori estimates already obtained in the previous two subsections.
 For a proof of $(b), (c)$ and $(d)$ we refer to our earlier work \cite{KW19} (see Lemma $6.1$ and Lemma $6.2$ for details).

\section{Global in time existence and free energy inequality}\label{section global existence}

\subsection{Convergence of MM-scheme}
Thanks to Lemma \ref{eti}(a) and (c) we can apply the refined Arzel\`{a}-Ascoli theorem (\cite[Proposition $3.3.1$]{AGS}) to the 
family $\{\bm{\rho}_{\tau}(t) | \ t \in [0,T], \ \tau \in (0,\tau^*)\}$ to obtain a curve $\bm{\rho}(t) :[0,T] \rightarrow (L^1(\rt))^n$ and a monotone decreasing
sequence $\tau_m \rightarrow 0+$ such that 

\vspace{0.2 cm}

\begin{itemize}
\item[{\bf (C1)}] $\rho_{\tau_m, i}(t) \rightharpoonup \rho_i(t)$ weakly in $L^1(\rt)$ for every $t \in [0,T],$ for all $i = 1,\ldots, n.$
\end{itemize}

\vspace{0.2 cm}

The continuity in time estimate Lemma \ref{eti}(c) also ensures that the limiting curve is H\"{o}lder continuous in time with respect to $\dw$ i.e., $\dws(\rho_i(t), \rho_i(s)) \leq \mathcal{C}|t-s|^{\frac{1}{2}}$ for all $i$.
Moreover, from the $L^2((0,T) \times \rt)$ bound of Lemma \ref{eti}(b) we infer that 

\vspace{0.2 cm}

\begin{itemize}
 \item[{\bf (C2)}]  $\rho_{\tau_m, i} \rightharpoonup \rho_i \ \ weakly \ in \ L^2((0,T)\times \rt)$ for all $i = 1,\ldots,n.$
\end{itemize}

\vspace{0.2 cm}

On the other hand, thanks to the Fisher information bound of Lemma \ref{eti}(b) and the compactness of vector fields Lemma \ref{cvf}
(see appendix) applied to the measures
$d\mu_m = (T\beta_i)^{-1}\rho_{\tau_m,i}dxdt,$ and the vector fields $\tilde v_m = (\nabla_x \rho_{\tau_m,i}/\rho_{\tau_m,i}, 1)
 =(v_m,1),$ we obtain the existence of a vector field $v_i \in L^2((0,T)\times\rt, \rho_i\ ;\rt)$ such that

\begin{align*}
 \lim_{m \rightarrow \infty}\int_0^T\inte \zeta \cdot v_m\rho_{\tau_m,i} \ dxdt= \int_0^T\inte \zeta \cdot v \rho_i \ dxdt, \ \ 
\ \ for \ all \ \zeta \in C_c^{\infty}((0,T)\times \rt ; \rt).
 \end{align*}
Moreover, since $v\rho_i \in L^1((0,T)\times \rt;\rt)$ we can identify $v_i = \frac{\nabla \rho_i}{\rho_i}$ 
through the following identity
\begin{align*}
\mbox{{\bf(C3)}} \ \  \lim_{m \rightarrow \infty}\int_0^T\inte \zeta \cdot v_m\rho_{\tau_m,i} 
 =\lim_{m \rightarrow \infty}-\int_0^T\inte (\nabla_x \cdot \zeta)  \rho_{\tau_m,i}
=-\int_0^T\inte (\nabla_x \cdot \zeta)  \rho_i,
\end{align*}

\vspace{0.2 cm}

which also ensures that $\rho_i \in L^1((0,T);W^{1,1}(\rt)).$ By lower semicontinuity
(Lemma \ref{cvf} equation \eqref{lscv}) and Lemma \ref{eti}(b)
\begin{align}\label{7.3}
 \int_0^T \inte \left|\frac{\nabla \rho_i(t)}{\rho_i(t)}\right|^2\rho_i(t) \ dxdt < \infty,
\end{align}
proving that $\frac{\nabla \rho_i}{\rho_i} \in L^2((0,T) \times \rt, \rho_i \ ; \rt)$ for all $i \in \{1,\ldots,n\}.$

In addition, using the $H^1_{loc}$ estimates of Lemma \ref{eti}(d), and proceeding as in \cite[Lemma $7.1$]{KW19} we get up to a subsequence
$u_{\tau_m, i} \rightarrow u_i$ strongly in $L^2((0,T); L^2_{loc}(\rt)),$ where $u_i$ is the Newtonian potential of the limiting curve $\rho_i.$
This together with the $L^2-H^2_{loc}$ estimate of Lemma \ref{eti}(d) and Simon's compactness results \cite[Lemma $9$]{Simon} ensures

\begin{equation*}
\noindent
\mbox{{\bf (C4)}} \hspace{2.8 cm}
u_{\tau_m,i} \rightarrow u_i \ \ strongly \ in \ L^2((0,T);H^1_{loc}(\rt)). \hspace{3 cm}
\end{equation*}

\subsection{Global in time existence of solutions}
\label{Gexistence} For any time dependent test function $\xi \in C_c^{\infty}([0,T) \times \rt),$ we apply 
Lemma \ref{el}(c) with the choices: $\vri = \rho_{\tau_m,i}^k, \ \eta_i = \rho_{\tau_m,i}^{k-1}, \ \psi(\cdot) = \xi(\cdot, (k-1)\tau_m),$
where $k$ is such that $k\tau_m \leq T.$ Summing over all $k$ satisfying $k \tau_m \leq T$ we obtain the following time discrete formulation 
of the system \eqref{kss}
\begin{align}\label{final el}
 -\int_0^{T} \inte \partial_t\xi\rho_{\tau_m,i} &- \inte \xi(0)\rho_i^0
 +\int_{0}^{T} \inte \nabla\xi \cdot \nabla \rho_{\tau_m,i} \notag \\
 &-\sj a_{ij}\int_{0}^{T} \inte  \nabla \xi \cdot \nabla u_{\tau_m,j}\rho_{\tau_m,i}dx 
= O(\tau_m^{\frac{1}{2}}),
\end{align}
for all $i = 1,\ldots, n.$ Now it is easy to pass the limit in \eqref{final el} using the convergence results {\bf(C1)} - {\bf(C4)}. For the first 
term we use {\bf(C2)}, for the third term we use {\bf(C3)} and finally for the last term we use {\bf (C2), (C4)} and the duality relation. The limiting equation is the 
weak formulation of \eqref{kss} stated in Definition \ref{weak sol}. Since $T$ is arbitrary, we conclude the proof of
global existence.

So far we have proved Theorem \ref{main} (a). The proof of (b) is the content of the next subsection.
\subsection{The free energy inequality} 
As in the sub-critical case, we show using De Giorgi variational interpolation, that the obtained solution 
satisfies the free energy inequality. 
Define for $\tau \in (0,\tau^*)$
\begin{align}\label{degiorgi}
\bm{\tilde \rho}_{\tau}(t) := \arg \min_{\bm{\rho}\in \ga} \left\{\f(\bm{\rho}) + \frac{1}{2(t - (k-1)\tau)}
\dw^2(\bm{\rho},\bm{\rho}_{\tau}^{k-1})\right\}, \ \ t \in ((k-1)\tau, k\tau]. 
\end{align}
With out loss of generality we can assume that $\bm{\tilde \rho}_{\tau}(k\tau) = \bm{\rho}_{\tau}^k.$ Note that since $0<t - (k-1)\tau \leq \tau< \tau^*$ 
and $\rka$ satisfies the inequality 
\begin{align*}
 \f(\rka) < \inf_{\bm{\rho} \in \ga} \F+ \frac{1}{2\tau}M_2(\rka) \leq \inf_{\bm{\rho} \in \ga} \F+ \frac{1}{2(t - (k-1)\tau)}M_2(\rka),
\end{align*}
the De Giorgi's interpolation \eqref{degiorgi} is well defined. Moreover, combining \cite[Theorem $3.1.4$ and Lemma $3.3.2$]{AGS} together with Lemma \ref{el}(b) we get the discrete energy identity stated below.

\begin{Lem}[Discrete energy identity] \label{dei}
 For every $k \in \mathbb{N}$ and $\tau \in (0,\tau^*)$ the De-Giorgi interpolation defined by \eqref{degiorgi} 
 satisfies the following energy identity:
\begin{align*}
&\si\frac{1}{2} \int_{0}^{k\tau}  \inte \left|\frac{\nabla \rho_{\tau,i}}{\rho_{\tau,i}} 
 - \sj a_{ij}\nabla u_{\tau,j}\right|^2 \rho_{\tau,i} \ dxdt \\
 &+ \si\frac{1}{2} \int_0^{k\tau}
 \inte \left|\frac{\nabla \tilde \rho_{\tau,i}}{\tilde \rho_{\tau,i}} - \sj a_{ij}\nabla \tilde u_{\tau,j}\right|^2 
 \tilde \rho_{\tau,i} \ dxdt + \f(\bm{\rho}_{\tau}(k\tau)) = \f(\bm{\rho}^0),
\end{align*}
where $\tilde u_{\tau,j}$ is the Newtonian potential associated to $\tilde \rho_{\tau,j}.$

Furthermore, for every $T>0$ there exists a constant $C(T)>0$ such that 
\begin{align}\label{same limit}
 \dw(\bm{\rho}_{\tau}(t), \bm{\tilde \rho}_{\tau}(t)) \leq C(T)\tau, \ \ for \ all \ t \in [0,T].
\end{align}
 \end{Lem}
As a consequence of \eqref{same limit} we infer
that $\bm{\rho}_{\tau_m}$ and $\bm{\tilde
	\rho}_{\tau_m}$ enjoys the same convergence
properties stated in
the previous section {\bf (C1)}-{\bf(C4)},
provided we establish the entropy, second
moment and the Fisher information bound
(Lemma \ref{eti}(a,b)). Here $\tau_m$ is the
same monotone
decreasing sequence used in the proof of
global in time existence. Moreover, both
$\bm{\rho}_{\tau_m}$ and $\bm{\tilde
	\rho}_{\tau_m}$ converge to the
same limit $\bm{\rho}.$
The finite Fisher information bound follows from
Lemma \ref{dei} and the inequality
\begin{align} \label{infm}
\frac{1}{2} \si \inte\frac{|\nabla\tilde\rho_{\tau,i}
	(t)|^2}{\tilde\rho_{\tau,i}(t)} \ dx
\leq & \si \inte \left|\frac{\nabla \tilde\rho_{\tau,i}
	(t)}{\tilde \rho_{\tau,i}(t)} - \sj a_{ij}\nabla \tilde
u_{\tau,j}(t)\right|^2
\tilde \rho_{\tau,i}(t) \ dx \notag \\
&+4n (\max_{i,j \in I} a_{ij})C_0\si \beta_i, \ \ for
\ all \ t \in [0,T],
\end{align}
where $C_0$ is a constant independent of the
interpolates. A proof of this inequality can be
found in \cite[Lemma $2.2$]{FM} (see also
\cite[Lemma $8.3$]{KW19}
in this setting). It remains to show that
\begin{align*}
\sup_m \sup_{t \in [0,T]} \left[\si \inte \tilde
\rho_{\tau_m, i}(t)|\ln \tilde \rho_{\tau_m, i}(t)| \
dx +
M_2( \bm{\tilde \rho}_{\tau_m}(t))\right] < \infty.
\end{align*}
Recall that we have all the estimates on the
interpolates $\bm{\rho}_{\tau_m}$ given by
Lemma \ref{eti}. In this sequel, any uniform
constant will be denoted by $\mathcal{C}.$

\noindent
{\it Second moment bound:} 
Let $t \in [0,T]$ and $k\in \mathbb{N}$ be such that $t \in ((k-1)\tau_m, k \tau_m]$. Then by \eqref{same limit} and Lemma \ref{apriori} we obtain the
required bound:
\begin{align} \label{dgn1}
\dw^2(\bm{\tilde \rho}_{\tau_m}(t),
\bm{\rho}_{\tau_m}^{k-1}) \leq
\mathcal{C}\tau_m, \ M_2(\bm{\tilde
	\rho}_{\tau_m}(t)) \leq \mathcal{C}.
\end{align}
\noindent
{\it Entropy bound:} Assume by contradiction,
there exists a sequence $t_m \in [0,T]$ such
that $\mathcal{H}(\bm{\tilde \rho}_{\tau_m}
(t_m)) \rightarrow \infty. $ Let $k_m \in
\mathbb{N}$ be such that $t_m \in ((k_m
-1)\tau_m, k_m\tau_m].$ Since the energy of
the interpolates $\f(\bm{\tilde \rho}_{\tau_m}
(t_m)) \leq \f(\bm{\rho}_{\tau_m}^{k_m-1})$ are
uniformly bounded, by Theorem \ref{mms
	existance thm} and Remark \ref{mms existence
	remark} we have $\lim_{m \rightarrow \infty}
\bm{\tilde \rho}_{\tau_m}(t_m)
=\bm{\beta}\bm{\delta}_0$ in the weak*
topology of measures.
In view of Lemma \ref{wdestimate} we get
\begin{align*}
M_2(\rkam) - M_2(\bm{\tilde \rho}_{\tau_m}
(t_m)) &= \dw^2(\bm{\tilde \rho}_{\tau_m}(t_m),
\rkam) \notag \\
M_2(\bm{\rho}^{k_m -2}_{\tau_m}) -
M_2(\bm{\rho}^{k_m -1}_{\tau_m}) &=
\dw^2(\bm{\rho}^{k_m-1}_{\tau_m},
\bm{\rho}^{k_m-2}_{\tau_m}) \notag \\
& \ \ \vdots \notag \\
M_2(\bm{\rho}^0_{\tau_m}) -
M_2(\bm{\rho}^1_{\tau_m}) &=
\dw^2(\bm{\rho}^1_{\tau_m},
\bm{\rho}^0_{\tau_m}).
\end{align*}
Adding all the terms and using \eqref{dgn1},
Lemma \ref{apriori} we see that
\begin{align} \label{contra}
M_2(\bm{\rho}^0) = M_2(\bm{\tilde
	\rho}_{\tau_m}(t_m)) + O(\tau_m).
\end{align}
On the other hand, invoking Lemma \ref{de la
	lemma} to the measure
\begin{align*}
\mu = \si \tilde \rho_{\tau_m,i}(t_m) +
\sum_{m=1}^{\infty}\frac{1}{k_m2^{m+1}}
\sum_{l = 0}^{k_m-1} \si \rho_{\tau_m, i}^l
\end{align*}
we obtain the existence of a convex function
$\vu$ satisfying the all the properties \eqref{p}
and such that
$\inte \vu(|x|^2)d \mu < \infty.$
Using Lemma \ref{better integrability} and
iterating the process to the interpolates
$\bm{\tilde \rho}_{\tau_m}(t_m), \bm{
	\rho}_{\tau_m}^{k_m-1}, \bm{
	\rho}_{\tau_m}^{k_m-2}, \ldots, \bm{
	\rho}_{\tau_m}^0$ we get
\begin{align}\label{contra2}
\si \inte \vu(|x|^2) \tilde \rho_{\tau_m,i}(t_m) \ dx
\leq e^{C_0T}\si \inte \vu(|x|^2) \rho_{i}^{0} \ dx
+ C_0Te^{C_0T}.
\end{align}
In view of \eqref{contra}, \eqref{contra2} and the
convergence to the Dirac mass we get a
contradiction. This establishes the uniform
entropy bound.

One more application of Lemma \ref{cvf} with the measures $\mu_m = (T\beta_i)^{-1}\rho_{\tau_m,i} dxdt$ and the vector fields $(v_m ,1)$ where
\begin{align*}
 v_{m} = \frac{\nabla \rho_{\tau_m,i}}{\rho_{\tau_m,i}} 
 - \sj a_{ij}\nabla u_{\tau_m,j}, 
\end{align*}
and using the semi-continuity \eqref{lscv} we obtain 
\begin{align} \label{ineq}
 &\int_{0}^{T}  \inte \left|\frac{\nabla \rho_{i}(t)}{\rho_{i}(t)} 
 - \sj a_{ij}\nabla u_{j}(t)\right|^2 \rho_{i}(t) \ dxdt \notag\\
 &\leq \liminf_{m\rightarrow +\infty}
 \int_{0}^{T}  \inte \left|\frac{\nabla \rho_{\tau_m,i}(t)}{\rho_{\tau_m,i}(t)} 
 - \sj a_{ij}\nabla u_{\tau_m,j}(t)\right|^2 \rho_{\tau_m,i}(t) \ dxdt.
\end{align}
$\bm{\tilde \rho}_{\tau_m}$ being converge to the same limit as $\bm{\rho}_{\tau_m},$ the inequality \eqref{ineq} holds for $\bm{\tilde \rho}_{\tau_m}$
as well. 
Passing to the limit in the discrete energy identity (Lemma \ref{dei})
and using the lower semi-continuity of $\f$ with respect to the weak $L^1$ convergence we get 
\begin{align*}
 \si\int_{0}^{T}  \inte \left|\frac{\nabla \rho_{i}(t)}{\rho_{i}(t)} 
 - \sj a_{ij}\nabla u_{j}(t)\right|^2 \rho_{i}(t) \ dxdt + \f(\bm{\rho}(T)) \leq \f(\bm{\rho}^0).
\end{align*}
This completes the proof of Theorem \ref{main}(b).
\section{Concentration does occur as time $t\nearrow\infty$} \label{section asymptotic}

This section divulges the behavior of the obtained solution $\bm{\rho}(\cdot, t)$ as time $t \nearrow \infty.$
Recall that $\bm{\rho}(\cdot, t)$ satisfies the free energy 
inequality 
\begin{align*}
 \f(\bm{\rho}(\cdot, t)) + \int_0^t \mathcal{D}_{\f}(\bm{\rho}(\cdot, s)) \ ds \leq \f(\bm{\rho}^0),
\end{align*}
for all $t \in [0, \infty)$ and consequently, the energy $\f(\bm{\rho}(\cdot, t))$ is bounded above by $\f(\bm{\rho}^0)$ for all $t \in [0, \infty).$ Moreover, it follows from the weak formulation (Definition \ref{weak sol}) that the second moment is conserved in time. Notice that the obtained solution in 
the previous section already satisfies the conservation of second moment. Indeed, applying de la Velle\'{e} Poussin's lemma, we confirm that $|x|^2\bm{\rho}_{\tau_m}(t)$ are equi-integrable, and hence the second moment is conserved: $M_2(\bm{\rho}(t)) = M_2(\bm{\rho}^0)$.
By Theorem \ref{mms existance thm} (and remark \ref{mms existence remark}), as $t_m \nearrow \infty$
either $\bm{\rho}(\cdot, t_m)$ converges to $\bm{\beta} \delta_{\bm{0}}$ or the entropy $\mathcal{H}(\bm{\rho}(\cdot, t_m))$ remains uniformly bounded.
In the following, we will show that the later situation is inconceivable.

\begin{Lem}\label{final concentration lemma}
 Assume $A$ and $\bm{\beta}$ satisfies the Assumption \ref{AssLambda} and $\bm{\rho}^0 \in \ga.$ Given any free energy solution $\bm{\rho}$ to \eqref{kss} we have 
 \begin{align*}
  \lim_{t \rightarrow \infty} \bm{\rho}(\cdot, t) = \bm{\beta} \delta_{\bm{0}},
 \end{align*}
where the convergence is in the sense of weak* convergence of measures.
\end{Lem}
\begin{proof}
 Assume by contradiction that there exists a sequence $t_m\nearrow \infty$ as $m \rightarrow \infty$ such that 
 \begin{align} \label{ent asump}
\sup_{m \in \mathbb{N}}\ \mathcal{H}(\bm{\rho}(\cdot, t_m)) < \infty.  
 \end{align}
 Passing to a subsequence
 we may assume that $t_{m+1} - t_m > 1$ for all $m.$ In the proof, we may have to pass to a further subsequence quite often, and for the simplicity of notations, 
 we will not distinguish between the original sequence and its subsequences. Since in the critical case the second moment is conserved in time, the measures $\bm{\rho}(\cdot, t_m)$ are tight.
 Moreover, since the entropy is also uniformly bounded, by Dunford-Pettis theorem there exists $\bm{\rho}^{\infty} \in \ga$ such that $\bm{\rho}(\cdot, t_m)
 \rightharpoonup \bm{\rho}^{\infty}(\cdot)$ weakly in $(L^1(\rt))^n.$ In addition, we have
 \begin{align}\label{fsm}
  0 < M_2(\bm{\rho}^{\infty}) \leq \liminf_{m \rightarrow \infty} M_2(\bm{\rho}(\cdot, t_m)) = M_2(\bm{\rho}^0) < \infty.
 \end{align}
 
 Let us set 
 \begin{align*}
  \bm{\rho}^m(x,t) = \bm{\rho}(x, t_m + t), \ \ for \ x \in \rt \ and \  t \in [0,1].
 \end{align*}

In the following we will show that $\bm{\rho}^m$ converges, in some sense, to a steady state of \eqref{kss} having finite second moment. To do that 
we need to obtain uniform estimates on entropy and the Fisher information of $\bm{\rho}^m$ all over again. Because all the estimates obtained earlier are local 
in time. Most importantly, we need to obtain uniform H\"{o}lder estimates in time. We know that for the solution obtained in subsection \ref{Gexistence}, $t \mapsto \bm{\rho}(\cdot, t)$ is $\frac{1}{2}$-H\"{o}lder continuous 
with respect to the 2-Wasserstein distance. Since we are considering any free energy solution (and we did not prove the uniqueness) we can not avail that information.
However, in the next few steps, we will 
show that if we replace the $2$-Wasserstein distance by $1$-Wasserstein distance then we have a global H\"{o}lder estimate. We divide the proof into several steps.
Any universal constant independent of $m$ will be denoted by $\mathcal{C}.$

\vspace{0.2 cm}

\noindent
{\bf Step 1:} Uniform $L^2$ and Fisher information bound:
\begin{align*}
 \si \int_0^1 \inte (\rho_i^m(x,t))^2 \ dxdt + \si \int_0^1 \inte \frac{|\nabla \rho_i^m(x,t)|^2}{\rho_i^m(x,t)} \ dxdt \leq \mathcal{C}. 
\end{align*}
Since $\bm{\beta}$ is critical $\f(\bm{\rho}(\cdot, t))$ is bounded from below uniformly with respect to $t.$ Using the free energy inequality we conclude that the dissipation 
of the free energy is integrable, i.e.,
\begin{align}\label{dof integrable}
 \lim_{t \rightarrow \infty} \int_0^t \mathcal{D}_{\f}(\bm{\rho}(\cdot, s)) \ ds \leq \f(\bm{\rho}^0) - \liminf_{t \rightarrow \infty}\f(\bm{\rho}(\cdot, t)) \leq \mathcal{C}.
\end{align}
As a consequence of \eqref{dof integrable} and the assumption $t_{m+1} - t_m >1$ for all $m$ we deduce
\begin{align} \label{identify 0 limit}
\lim_{m \rightarrow \infty} \int_{t_m}^{t_m + 1} \mathcal{D}_{\f}(\bm{\rho}(\cdot, s)) \ ds =
\lim_{m \rightarrow \infty} \int_{0}^{1} \mathcal{D}_{\f}(\bm{\rho}^m(\cdot, s)) \ ds = 0.
\end{align}
The inequality \eqref{identify 0 limit} together with \eqref{infm} establishes the Fisher information bound 
\begin{align*}
 \si \int_0^1 \inte \frac{|\nabla \rho_i^m(x,t)|^2}{\rho_i^m(x,t)} \ dxdt \leq \mathcal{C}.
\end{align*}
By \cite[Lemma $2.1$]{FM} any $L^p$-norm can be controlled by the Fisher information and in particular
\begin{align*}
  \si \int_0^1 \inte (\rho_i^m(x,t))^2 \ dxdt \leq \mathcal{C}.
\end{align*}

\vspace{0.2 cm}

\noindent
{\bf Step 2.} Uniform H\"{o}lder estimate: $ \dwo(\bm{\rho}(\cdot, t_0), \bm{\rho}(\cdot, t_1)) \leq \mathcal{C}|t_0 - t_1|^{\frac{1}{2}},$ for all $t_0,t_1 \in [0, \infty).$

\vspace{0.2 cm}

Fix $t_0 < t_1.$ For simplicity of presentation we abbreviate the equation \eqref{kss} as $\partial_t \rho_i = \nabla \cdot (\rho_i \nabla \frac{\partial \f 
(\bm{\rho}(\cdot, t))}{\partial \rho_i}).$ 
So that the weak formulation can be reformulated as 
\begin{align} \label{qa}
 \int_{t_0}^{t_1}\inte \partial_t \xi \rho_i = \int_{t_0}^{t_1} \inte \nabla_x \xi \cdot \left(\nabla_x\frac{\partial \f (\bm{\rho}(\cdot, t))}{\partial \rho_i}\right) \rho_i.
\end{align}
for every $\xi \in C_c^1((t_0,t_1) \times \rt).$
We apply test functions of the form $\xi(x,t) = h(t)\psi(x)$ in \eqref{qa}, where $h \in C_c^1(t_0,t_1)$ and $\psi \in C_c^1(\rt)$ and set $f_i(t) = \inte \psi(x) \rho_i(x,t) \ dx.$
Then we can rewrite \eqref{qa} 
\begin{align*}
 \int_{t_0}^{t_1} h^{\prime}(t)f_i(t) \ dt = \int_{t_0}^{t_1} h(t) \inte \nabla_x \psi \cdot \left(\nabla_x\frac{\partial \f (\bm{\rho}(\cdot, t))}{\partial \rho_i}\right) \rho_i.
\end{align*}
Hence $f_i^{\prime}(t) = -\inte\nabla_x \psi \cdot \left(\nabla_x\frac{\partial \f (\bm{\rho}(\cdot, t))}{\partial \rho_i}\right) \rho_i$ and moreover, 
$L^2$-norm of $f_i^{\prime}$ can be estimated as 
\begin{align*}
 \int_{t_0}^{t_1} |f_i^{\prime}(t)|^2 \ dt &=  \int_{t_0}^{t_1} \left(\inte \nabla_x \psi \cdot \left(\nabla_x\frac{\partial 
 \f (\bm{\rho}(\cdot, t))}{\partial \rho_i}\right) \rho_i\right)^2 \ dt \\
 &\leq ||\nabla \psi||_{L^{\infty}}^2  \int_{t_0}^{t_1} \left(\inte \left|\nabla \frac{\partial \f (\bm{\rho}(\cdot, t))}{\partial \rho_i}\right|^2 \rho_i\right)
 \left(\inte \rho_i\right) \ dt \\
 & \leq \beta_i||\nabla \psi||_{L^{\infty}}^2 \left(\int_0^{\infty}\mathcal{D}_{\f}(\bm{\rho}(\cdot, t)) \ dt\right)\\
&\leq \mathcal{C}||\nabla \psi||_{L^{\infty}}^2 .
 \end{align*}
Since $\rho_i\in L^2_{loc}((0,\infty); L^2(\rt)),$ we also deduce that $f_i \in L^2(t_0,t_1)$ and hence $f_i \in W^{1,2}(t_0,t_1).$ As a consequence, 
$f_i$ are absolutely continuous and by fundamental theorem of calculus
\begin{align}\label{take sup}
 |f_i(t_0) - f_i(t_1)| \leq \int_{t_0}^{t_1} |f_i^{\prime}(t)| \ dt \leq ||f_i^{\prime}||_{L^2(t_0,t_1)} |t_0 - t_1|^{\frac{1}{2}} \leq \mathcal{C} ||\nabla \psi||_{L^{\infty}} 
 |t_0 - t_1|^{\frac{1}{2}}.
\end{align}
Taking the supremum in \eqref{take sup} over all $\psi \in C_c^1(\rt)$ satisfying $||\nabla \psi||_{\infty} \leq 1$ and using Kantorovich duality \eqref{w1dual} we get the desired result: 
\begin{align*}
 \dwso(\rho_i(\cdot, t_0), \rho_i(\cdot, t_1)) = \beta_i^{-\frac{1}{2}} \sup_{||\nabla \psi||_{L^{\infty}} \leq 1} |f_i(t_0) - f_i(t_1)| \leq \mathcal{C}|t_0 - t_1|^{\frac{1}{2}}.
\end{align*}

Applying step 2 to the sequence $\bm{\rho}^m$ we obtain $\dwo(\bm{\rho}^m(\cdot, t), \bm{\rho}^m(\cdot, s)) \leq \mathcal{C}|t-s|^{\frac{1}{2}}$ for all $s,t \in [0,1].$

\vspace{0.2 cm}

\noindent
{\bf Step 3:} Uniform entropy bound: $\sup_{t \in [0,1]} \mathcal{H}(\bm{\rho}^m(\cdot, t)) \leq \mathcal{C}.$

\vspace{0.2 cm}

We evaluate the time derivative of the entropy
\begin{align*}
 \frac{d \mathcal{H}(\bm{\rho}_i^m(\cdot, t))}{dt} &= \inte (1 + \ln \rho_i^m) \partial_t \rho_i^m \ dx\\
 &= - \inte \nabla \rho_i^m \cdot \left(\nabla \frac{\partial \f(\bm{\rho}^m(\cdot, t))}{\partial \rho_i}\right) \ dx \\
 & \leq \frac{1}{2} \inte \frac{|\nabla \rho_i^m|^2}{\rho_i^m} + \frac{1}{2}\inte \left| \nabla\frac{\partial \f(\bm{\rho}^m(\cdot, t))}{\partial \rho_i}\right|^2\rho_i^m.
\end{align*}
For $t \in [0, 1]$ integrating the above inequality from $0$ to $t$ we get
\begin{align*}
 \mathcal{H}(\rho_i^m(\cdot, t)) \leq \mathcal{H}(\rho_i^m(\cdot, 0)) + \frac{1}{2}\int_0^1\inte \frac{|\nabla \rho_i^m|^2}{\rho_i^m} \ dx ds + \frac{1}{2}\int_0^{\infty}
 \mathcal{D}_{\f}(\bm{\rho}(\cdot, s)) \ ds \leq \mathcal{C}.
\end{align*}
In the last inequality we have used the hypothesis \eqref{ent asump} and step 1.

In view of step $1$ and step $3$ we also have 

\vspace{0.2 cm}

\noindent
{\bf Step 4:} For any $R > 0$ the Newtonian potential $u_i^m(x,t) := -\frac{1}{2\pi} \ln |\cdot| \star \rho_i^m(\cdot, t)$ satisfies
\begin{align*}
 \sup_{t \in [0,1]} ||u_i^m(\cdot, t)||_{H^1(B_R)} + \int_0^1 ||u_i^m(\cdot, t)||_{H^2(B_R)}^2 \ dt \leq \mathcal{C}, \ for \ all \ i \in I.
\end{align*}

It is essentially Lemma $4.4$ and Lemma $6.1$ of \cite{KW19}. We skip the proof and refer the reader to \cite{KW19} for details. 

\vspace{0.2 cm}

\noindent
{\bf Step 5:} Identifying the limit.

\vspace{0.2 cm}

We define $\mu_i^m = \beta_i^{-1} \rho_i^m(x,t) dxdt$ and $v_i^m (x,t) = \frac{\nabla \rho_i^m(x,t)}{\rho_i^m} - \sj a_{ij} \nabla u_j^m(x,t)$ for $x \in \rt, t \in [0,1]$
and $\tilde v_i^m = (v_i^m, 1).$
By step $1,$ $\sup_m ||\tilde v_i^m||_{L^2((0,1)\times \rt, \mu_i^m; \mathbb{R}^3)} < \infty$ and moreover, using \eqref{identify 0 limit} we see that 
\begin{align}\label{os0}
 || v_i^m||_{L^2((0,1)\times \rt, \mu_i^m; \mathbb{R}^2)} \rightarrow 0, \ \ as \ m \rightarrow \infty.
\end{align}

In view of step $1$ - step $4,$ we can invoke refined Arzel\'{a}-Ascoli's lemma \cite[Proposition $3.3.1$]{AGS} and the arguments used in section \ref{section global existence},
to conclude that $\rho_i^m(\cdot,t) \rightharpoonup \rho_i^{\infty}(\cdot, t)$ weakly in $L^1(\rt)$ for every $t \in [0,1],$ $\rho_i^m \rightharpoonup \rho_i^{\infty}$ weakly in $L^2((0,1)\times \rt)$ and $u_i^m \rightarrow u_i^{\infty}$
strongly in $L^2((0,1); H^1_{loc}(\rt)).$ Applying Lemma \ref{cvf} to $\mu_i^m$ and $v_i^m$ we find the existence of a vector field  
$v_i^{\infty} \in L^2((0,1)\times \rt, \rho_i^{\infty};\rt)$ such that
\begin{align*}
 \int_0^1 \inte \zeta \cdot v_i^m \rho_i^m \ dxdt \rightarrow \int_0^1 \inte \zeta \cdot v_i^{\infty} \rho_i^{\infty} \ dxdt, \ \ for \ all \ \zeta \in 
 C_c^{\infty}((0,1)\times \rt).
\end{align*}

Taking into account the convergence results mentioned above, we can identify the vector fields $v_i^{\infty}$ through
\begin{align}\label{os}
 \int_0^1\inte \zeta \cdot v_i^{\infty} \rho_i^{\infty} = \lim_{m \rightarrow +\infty}\int_0^1\inte \zeta \cdot v_i^m\rho_i^m 
= \int_0^1\inte \zeta \cdot \left(\nabla \rho_{i}^{\infty}
 - \sj a_{ij}\nabla u_{j}^{\infty} \rho_{i}^{\infty}\right).
\end{align}
On the other hand, as a consequence of \eqref{os0} $v_i^{\infty} \equiv 0$ and hence from \eqref{os} we conclude
\begin{align} \label{concluding}
 \int_0^1\inte \zeta \cdot \left(\nabla \rho_{i}^{\infty}
 - \sj a_{ij}\nabla u_{j}^{\infty} \rho_{i}^{\infty}\right) = 0, \ \ for \ all \ \zeta \in C_c^{\infty}((0,1)\times \rt).
\end{align}

\vspace{0.2 cm}

\noindent 
{\bf Step 6:} Concluding the proof.

\vspace{0.2 cm}

The equation \eqref{concluding} is equivalent to saying $\rho_i^{\infty}(x,t) = \frac{\beta_i e^{\sj a_{ij}u_j^{\infty}(x,t)}}{\inte e^{\sj a_{ij}u_j^{\infty}(z,t)} \ dz}$
and $u_i^{\infty}$ satisfies the Liouville system 
\begin{align*}
 -\Delta_x u_i^{\infty}(x,t) = \frac{\beta_i e^{\sj a_{ij}u_j^{\infty}(x,t)}}{\inte e^{\sj a_{ij}u_j^{\infty}(z,t)} \ dz} \ \ \ \ in \ \rt.
\end{align*}

It follows from a result of Chipot, Shafrir and the second author \cite[Lemma $3.1$, Proposition $3.1$]{CSW} that $u_i^{\infty}$ has the asymptotic behaviour
\begin{align*}
 \left|\sj a_{ij}u_j^{\infty}(x,t) + \frac{1}{2\pi} \sj a_{ij}\beta_j \ln |x|\right| = O(1) \ \ \ \ if \ |x| > R \ is \ large. 
\end{align*}
As a result
\begin{align*}
 \inte |x|^2\rho_i^{\infty}(x,0) \ dx  &= \inte |x|^2  \frac{\beta_i e^{\sj a_{ij}u_j^{\infty}(x,0)}}{\inte e^{\sj a_{ij}u_j^{\infty}(z,0)} \ dz} \ dx \\
 &  \geq e^{O(1)}\int_{\{|x| \geq R\}} |x|^{(2 - \frac{1}{2\pi}\si a_{ij}\beta_j)} \ dx.
\end{align*}
Since, by \eqref{fsm}, the second moment is finite, we must have $2 - \frac{1}{2\pi}\si a_{ij}\beta_j < -2$ i.e., $\sj a_{ij}\beta_j > 8\pi$ for all $i \in I.$
But according to our assumption 
\begin{align*}
 0 = \Lambda_I(\bm{\beta}) = \si \beta_i (8 \pi - \sj a_{ij}\beta_j) < 0.
\end{align*}
This contradiction assures \eqref{ent asump} is not possible. In view of the energy bound and Theorem \ref{mms existance thm}, remark \ref{mms existence remark}, $\bm{\rho}^m$ must concentrate, and hence the proof of the theorem is completed.
\end{proof}

\section{Appendix}\label{section appendix}

This last section has been devoted to the proof of Lemma \ref{better integrability}, and further we recall a compactness result which has been used frequently in this article.
Recall that $\Psi_R$ is the smooth cut off function introduced in \eqref{cut off} and $\vu$ is the de la Vell\'{e}e Poussin convex function satisfying all 
the properties in \eqref{p}. Furthermore, we have the following estimates: 
the concavity of $\vu^{\prime}$ and $\vu^{\prime}(0) = 0$ implies
\begin{align*}
r\vu^{\prime \prime}(r) \leq \frac{\vu(r)}{r}.
\end{align*}
By smoothness of $\vu,$ it is not difficult to see that 
\begin{align*}
\vu^{\prime}(r) \leq 2\frac{\vu(r)}{r} \leq 2\left(c_1 \vu(r) + c_2\right).
\end{align*}
In the proof we are going to use these estimates frequently. Any universal constant will be denoted by 
$C_1, C_2.$

\vspace{0.2 cm}

\noindent
{\bf Proof of Lemma \ref{better integrability}:}
\begin{proof}
	As before the central idea is to use the test function $\psi_R(x) = \vu(|x|^2)\Psi_R(x)$ in the Euler-Lagrange equation Lemma \ref{el}(d). We estimate one by one:
	\begin{align*}
	&\nabla \psi_R(x) = 2x \vu^{\prime}(|x|^2) \Psi_R(x) + \vu(|x|^2) \nabla \Psi_R(x),\\
	&\Delta \psi_R(x) = 4 \vu^{\prime}(|x|^2) \Psi_R(x) + 4|x|^2\vu^{\prime\prime}(|x|^2)\Psi_R(x) \\
	& \hspace{5 cm} + 4\vu^{\prime}(|x|^2) (x \cdot \nabla\Psi_R(x)) + \vu(|x|^2)\Delta \Psi_R(x). 
	\end{align*}
	It follows from the properties of $\vu$ mentioned above and that $\sup_x |x\nabla \Psi_R(x)| = O(1)$
	\begin{align*}
	|\Delta \psi_R(x)| \leq C_1 \vu(|x|^2) + C_2.
	\end{align*}
	As a result 
	\begin{align*}
	\Big|\si \inte \Delta\psi_R(x)\vri (x)\Big| \leq C_1 \si \inte \vu(|x|^2)\vri + C_2.
	\end{align*}	
	Now we can write
	\begin{align*}
	\inte \inte \frac{(\nabla \psi_R(x) - \nabla \psi_R(y))\cdot(x-y)}{|x-y|^2} \vri(x)\vrj(j) \ dxdy = I_1 + I_2 + I_3 + I_4
	\end{align*}
	where 
	\begin{align*}
	&I_1 = \inte \inte \frac{2 \left(x \vu^{\prime}(|x|^2) - y \vu^{\prime}(|y|^2)\right)\cdot (x-y)}{|x-y|^2}\Psi_R(x) \vri(x)\vrj(y) \ dxdy, \\
	&I_2 = \inte \inte \frac{2y \cdot (x-y)}{|x-y|^2}\vu^{\prime}(|y|^2)(\Psi_R(x) - \Psi_R(y))\vri(x)\vrj(y)\ dxdy, \\
	&I_3 = \inte \inte \frac{(\nabla \Psi_R(x) - \nabla \Psi_R(y)) \cdot (x-y)}{|x-y|^2}\vu(|x|^2)\vri(x)\vrj(y) \ dxdy, \\
	&I_4 = \inte \inte \frac{(x-y)\cdot \nabla \Psi_R(y) }{|x-y|^2} (\vu(|x|^2) - \vu(|y|^2))\vri(x)\vrj(y) \ dxdy.
	\end{align*}
	
	The integrals $I_2$ and $I_3$ are easy to estimate:
	\begin{align*}
	|I_2 + I_3| \leq C_1 \si \inte \vu(|x|^2)\vri + C_2.
	\end{align*}
	To estimate $I_4$ we use the following 
	\begin{align*}
	|\vu(|x|^2) - \vu(|y|^2)| &= |\int_0^1 \frac{d}{ds} \vu(|sx + (1-s)y|^2)| \ ds\notag \\
	&\leq 2|x-y|\int_0^1 |sx + (1-s)y| \vu^{\prime}(|sx + (1-s)y|^2) \ ds.
	\end{align*}
	As a result 
	\begin{align*}
	& |I_4| \leq 2\inte \inte \int_0^1 |sx + (1-s)y| \vu^{\prime}(|sx + (1-s)y|^2)|\nabla \Psi_R(y)|\vri(x)\vrj(y) \ dsdxdy \notag \\
	&\leq 2\int_0^1 \int\int_{\{|sx + (1-s)y| \leq 1\}} |sx + (1-s)y| \vu^{\prime}(|sx + (1-s)y|^2)|\nabla \Psi_R(y)|\vri(x)\vrj(y) \ dxdyds \notag \\
	& \ \ + 2\int_0^1 \int\int_{\{|sx + (1-s)y| > 1\}} |sx + (1-s)y| \vu^{\prime}(|sx + (1-s)y|^2)|\nabla \Psi_R(y)|\vri(x)\vrj(y) \ dxdyds \notag \\
	&\leq 2\int_0^1 \inte\inte \vu^{\prime}(|sx + (1-s)y|^2)|\nabla \Psi_R(y)|\vri(x)\vrj(y) \ dxdy ds\notag \\
	& \ \ +4\int_0^1 \int\int_{\{|sx + (1-s)y| > 1\}} \vu(|sx + (1-s)y|^2)|\nabla \Psi_R(y)|\vri(x)\vrj(y) \ dxdy ds\notag \\
	&\leq C_1\int_0^1 \inte \inte \vu(|sx + (1-s)y|^2)|\nabla \Psi_R(y)|\vri(x)\vrj(y) \ dxdyds + C_2\notag \\
	&\leq C_1\si \inte \vu(|x|^2)\vri(x) \ dx + C_2.
	\end{align*}
	In the third inequality we have used $s \vu^{\prime}(s^2) \leq2 s \frac{\vu(s^2)}{s^2} \leq 2\vu(s^2)$ provided $s>1.$ The last inequality follows from the convexity of $|x|^2$ and $\vu$ and the monotonicity of $\vu.$ Finally, to estimate $I_1$ we use the following
	\begin{align*}
	\left(x\vu^{\prime}(|x|^2) - y \vu^{\prime}(|y|^2)\right)\cdot (x-y) = & \ \frac{1}{2}|x-y|^2(\vu^{\prime}(|x|^2) + \vu^{\prime}(|y|^2)) \\
	&+ \frac{1}{2}(|x|^2 - |y|^2) \left(\vu^{\prime}(|x|^2) - \vu^{\prime}(|y|^2)\right) \\
	\geq & \ \frac{1}{2}|x-y|^2(\vu^{\prime}(|x|^2) + \vu^{\prime}(|y|^2)),
	\end{align*}
	where in the second line we used the convexity of $\vu.$ Using $\psi_R$ as a test function, the right hand side of the Euler-Lagrange equation Lemma \ref{el}(d) 
	can be estimated as follows:
	\begin{align*}
	&\si\inte  \Delta \psi_R(x) \vri(x)dx - \si\sj \frac{a_{ij}}{4\pi} \inte \inte  \frac{(\nabla \psi_R(x) - \nabla \psi_R(y))\cdot (x-y)}{|x-y|^2}\vri(x)\varrho_j(y)dxdy \notag \\
	&\leq  \ C_1\si \inte \vu(|x|^2)\vri(x) \ dx + C_2 \notag \\
	& \ -\si \sj \frac{a_{ij}}{4\pi}\inte \inte \frac{1}{2}(\vu^{\prime}(|x|^2) + \vu^{\prime}(|y|^2))\Psi_R(x)\vri(x) \vrj(y) \ dxdy \notag\\
	&\leq  \ C_1\si \inte \vu(|x|^2)\vri(x) \ dx + C_2
	\end{align*}
	While the left hand side of Lemma \ref{el}(d) can be written as 
	\begin{align} \label{lhs of el}
	&\frac{1}{\tau}\si\inte (\nabla \varphi_i(x)-x) \cdot  \nabla\psi_R(\nabla\varphi_i(x))\eta_i(x)dx \notag \\
	=& \ \frac{1}{\tau} \left[\si \inte (x \cdot \nabla\psi_R(x))\vri(x)\ dx  - \si \inte (x \cdot \nabla \psi_R(\nabla\varphi_i(x))\eta_i(x) \ dx \right] \notag \\
	= & \ \frac{2}{\tau} \left[ \si \inte |x|^2 \vu^{\prime}(|x|^2) \Psi_R(x) \vri(x) \ dx  \right. \notag\\
	&\left. \ \ \ \ \ \ \ \ \ \ \ \ - \si \inte (x \cdot \nabla\varphi_i(x)) \vu^{\prime}(|\nabla \varphi_i(x)|^2) \Psi_R(\nabla \varphi_i(x)) \eta_i(x) \ dx \right] \notag \\
	& \ +\frac{1}{\tau}\left[\si \inte (x \cdot \nabla \Psi_R(x)) \vu(|x|^2) \vri(x) \ dx \right. \notag\\
	&\left. \ \ \ \ \ \ \ \ \ \ \ \ - \si \inte (x \cdot \nabla \Psi_R(\nabla \varphi_i(x)) \vu(|\nabla \varphi_i(x)|^2) \eta_i(x) \ dx\right] \notag \\
	= & \ \frac{2}{\tau} \left[ \si \inte |x|^2\vu^{\prime}(|x|^2) \vri(x) \ dx 
	- \si \inte (x \cdot \nabla\varphi_i(x)) \vu^{\prime}(|\nabla \varphi_i(x)|^2) \eta_i(x) \ dx \right] + o(1)\notag \\
	\end{align}
	as $R \rightarrow \infty.$ In the last line we have used the integrability assumption $\si \vu(|\cdot|^2)(\eta_i + \vri) \in L^1(\rt)$ and the dominated convergence theorem. The justification of passing to the limit will be clear in a moment (see \eqref{e}, \eqref{r} below). For the time being note that $|(x \cdot \nabla\varphi_i(x)) \vu^{\prime}(|\nabla \varphi_i(x)|^2) \Psi_R(\nabla \varphi_i(x))| \leq |x||\nabla \varphi_i(x)|^2|\nabla \Psi_R(\nabla \varphi_i(x))| \frac{\vu(|\nabla \varphi_i(x)|^2)}{|\nabla \varphi_i(x)|^2} = O(|x||\nabla \varphi_i(x)| \vu^{\prime}(|\nabla \varphi_i(x)|^2)).$
	As a consequence, we deduce from the Euler-Lagrange equation with $\psi_R$ as a test function
	\begin{align} \label{t}
	2\si \inte |x|^2\vu^{\prime}(|x|^2) \vri(x) \ dx \leq& \ 2\inte (x \cdot \nabla\varphi_i(x)) \vu^{\prime}(|\nabla \varphi_i(x)|^2) \eta_i(x) \ dx \notag \\ 
	&+ \tau \left(C_1\si \inte \vu(|x|^2)\vri(x) \ dx + C_2\right)
	\end{align}
	Using the inequality $2a \cdot b \leq |a|^2 + |b|^2$ we can estimate the first term on the right hand side of \eqref{t}
	as follows:
	\begin{align} \label{e}
	2\inte (x \cdot \nabla\varphi_i(x)) \vu^{\prime}(|\nabla \varphi_i(x)|^2) \eta_i(x) \ dx 
	\leq & \  \si \inte |\nabla \varphi_i(x)|^2\vu^{\prime}(|\nabla \varphi_i(x)|^2) \eta_i(x) \ dx \notag \\
	&+  \si \inte |x|^2 \vu^{\prime}(|\nabla \varphi_i(x)|^2)\eta_i(x) \ dx. \notag \\
	= & \  \si \inte |x|^2\vu^{\prime}(|x|^2) \vri(x) \ dx \notag \\
	&+  \si \inte |x|^2 \vu^{\prime}(|\nabla \varphi_i(x)|^2)\eta_i(x) \ dx.
	\end{align}
	Let $\vu^*$ be the conjugate convex function to $\vu.$ 
	Then the last term in \eqref{e} can be bound from above by
	\begin{align} \label{r}
	&\si \inte |x|^2 \vu^{\prime}(|\nabla \varphi_i(x)|^2)\eta_i(x) \ dx \notag \\
	\leq& \ 
	\si \inte \vu(|x|^2) \eta_i(x) \ dx + \si \inte  \vu^*(\vu^{\prime}(|\nabla \varphi_i(x)|^2))\eta_i(x) \ dx \notag \\
	=& \ \si \inte \vu(|x|^2) \eta_i(x) \ dx + \si \inte  \vu^*(\vu^{\prime}(|x|^2))\vri(x) \ dx
	\end{align}
	The properties mentioned in \eqref{p} ensures that $\vu^*(\vu^{\prime}(r)) \leq \vu(r)$ for every $r >0$ (see \cite[Lemma $B.1$]
	{LMdela}). As a result, all the terms in \eqref{r} are finite, which also justifies the passing to the limit in \eqref{lhs of el}.
	Plugging \eqref{e} and \eqref{r} into \eqref{t} we get
	\begin{align}\label{last line lemma}
	\si \inte \left[|x|^2\vu^{\prime}(|x|^2) - \vu^*(\vu^{\prime}(|x|^2)) \right]\vri(x) \ dx \leq& \si \inte \vu(|x|^2) \eta_i(x) \ dx \notag \\
	&+ \tau \left(C_1\si \inte \vu(|x|^2)\vri(x) \ dx + C_2\right)
	\end{align}
	Using the properties of $\vu$ mentioned in \eqref{p} it can also be shown that  $r\vu^{\prime}(r) -\vu^* (\vu^{\prime}(r)) = \vu(r)$ (see \cite[Lemma $B.1$]{LMdela}). Thus we obtain from \eqref{last line lemma}
	\begin{align*}
	(1- C_1\tau) \si \inte \vu(|x|^2) \vri(x) \ dx \leq& \si \inte \vu(|x|^2) \eta_i(x) \ dx + C_2\tau,
	\end{align*}
	which is equivalent to the result claimed in the lemma.
\end{proof}

We have used the following compactness of vector fields result whose proof can be found in 
\cite[Theorem $5.4.4$]{AGS}:
\begin{lema}[Compactness of vector fields]\label{cvf}
	Let $\Omega$ be an open set in $\mathbb{R}^N$. If $\{\mu_m\}_m$ is a sequence of probability measures in $\Omega$
    converging to $\mu$ in the weak * topology of measures and $\{v_m\}_m$
	is a sequence of vector fields in $L^2(\Omega, \mu_m;\mathbb{R}^N)$ satisfying 
	\begin{align*}
	\sup_m ||v_m||_{L^2(\Omega,\mu_m;\mathbb{R}^N)} < +\infty,
	\end{align*}
	then there exists a vector field $v \in L^2(\Omega, \mu;\mathbb{R}^N)$ such that 
	\begin{align*}
	\lim_{m\rightarrow \infty} \int_{\Omega} \zeta \cdot v_m \ d\mu_m = \int_{\Omega} \zeta \cdot v \ d\mu, 
	\ \ for \ all \ \zeta 
	\in C_c^{\infty}(\Omega; \mathbb{R}^N)
	\end{align*}
	and satisfy 
	\begin{align} \label{lscv}
	||v||_{L^2(\Omega, \mu; \mathbb{R}^N)} \leq \liminf_{m\rightarrow \infty} ||v_m||_{L^2(\Omega, \mu_m; \mathbb{R}^N)}.
	\end{align}
\end{lema}

\vspace{0.2 cm}

\noindent
{\it Acknowledgement.} D. Karmakar acknowledges the support of the Department of Atomic Energy, Government of India, under project no. 12-R$\&$D-TFR-5.01-0520.
Part of this research work was completed during the stay of the first author at the Technion as a post-doctoral fellow. He deeply acknowledges the warm hospitality, friendly environment and partial financial support from Technion fellowship. 

Part of this research was completed while the second author was on sabbatical at the  Emory U., USA.

 \label{Bibliography}

\bibliography{KS_system_Optimal_Transport} 

\bibliographystyle{alpha} 

\end{document}